\documentclass[reqno]{amsart}
\usepackage{relsize}
\usepackage{scalerel}
\usepackage{stackengine,wasysym}
\usepackage{todonotes}
\usepackage{xcolor}
\usepackage{mathrsfs}
\usepackage{dsfont}
\usepackage{mathtools}
\usepackage[sort,nocompress]{cite}
\usepackage[colorlinks,
            linkcolor=blue,
            anchorcolor=blue,
			citecolor=blue
]{hyperref}

\usepackage{amsmath,amssymb} 
\usepackage[toc,page]{appendix}
\usepackage{bbm}
\usepackage{verbatim}
\allowdisplaybreaks
\theoremstyle{plain}
\newtheorem{lemma}{Lemma}[section]
\newtheorem{theorem}[lemma]{Theorem}

\newtheorem{proposition}[lemma]{Proposition}
\newtheorem{definition}[lemma]{Definition}

\newtheorem{example}{Example}

\theoremstyle{remark}

\newcommand*  {\N} {{\mathbb N}}

\newcommand{\Rb}{\mathbb{R}}
\newcommand{\Nb}{\mathbb{N}}

\newcommand{\Eb}{\mathbb{E}}
\newcommand{\Pb}{\mathbb{P}}
\newcommand{\Hb}{\mathbb{H}}

\newcommand{\Fb}{\mathbb{F}}
\newcommand{\Fc}{\mathcal{F}}
\newcommand{\Fct}{\left( \mathcal{F}_t \right)_{t \geq 0}}

\newcommand{\s}{\sigma}
\newcommand{\Qb}{\mathbb{Q}}
\newcommand{\Tb}{\mathbb{T}}
\newcommand{\Wb}{\mathbb{W}}
\newcommand{\Pc}{\mathcal{P}}
\newcommand{\LL}{\Lambda}
\newcommand{\bV}{\overline{V}}

\newcommand{\Uc}{\mathscr{U}}

\newcommand{\Sc}{\mathcal{S}}

\newcommand{\bx}{x}
\newcommand{\bk}{k}

\newcommand{\hook}{\hookrightarrow}	

\makeatletter
\newcommand*{\rom}[1]{\expandafter\@slowromancap\romannumeral #1@}
\makeatother

\def\a{\alpha}
\def\b{\beta}

\numberwithin{equation}{section}

\begin{document}

\vskip 0.125in

\title[Stochastic Primitive Equations]
{On the local well-posedness of fractionally dissipated primitive equations with transport noise}

\date{\today}

\author[R. Hu]{Ruimeng Hu}
\address[R. Hu]
{	Department of Mathematics \\
Department of Statistics and Applied Probability \\
     University of California  \\
	Santa Barbara, CA 93106, USA.} \email{rhu@ucsb.edu}

\author[Q. Lin]{Quyuan Lin}
\address[Q. Lin]
{	School of Mathematical and Statistical Sciences \\
Clemson University\\
Clemson, SC 29634, USA.} \email{quyuanl@clemson.edu}

\author[R. Liu]{Rongchang Liu}
\address[R. Liu]{Department of Mathematics, University of Arizona, Tucson, AZ 85721, USA}
\email{lrc666@arizona.edu}

\begin{abstract}
    We investigate the three-dimensional fractionally dissipated primitive equations with transport noise, focusing on subcritical and critical dissipation regimes characterized by \( (-\Delta)^{s/2} \) with \( s \in (1,2) \) and \( s = 1 \), respectively. For $\sigma>3$, we establish the local existence of unique pathwise solutions in Sobolev space $H^\sigma$. This result applies to arbitrary initial data in the subcritical case ($s \in(1,2)$), and to small initial data in the critical case ($s=1$). The analysis is particularly challenging due to the loss of horizontal derivatives in the nonlinear terms and the lack of full dissipation. To address these challenges, we develop novel commutator estimates involving the hydrostatic Leray projection.    
\end{abstract}

\maketitle


\vskip0.1in


\section{Introduction} 
In this paper, we establish the local existence and uniqueness of pathwise solutions in Sobolev spaces for the following three-dimensional fractionally dissipated primitive equations (PE, also called the hydrostatic Navier-Stokes equations) with Stratonovich transport noise, defined on $\mathbb T^3 = \mathbb R^3/\mathbb Z^3$, 
\begin{align}\label{PE-system}
    \begin{split}
        &d V + (V\cdot \nabla_h V + w\partial_z V + f_0 V^\perp +\nabla_h p ) dt = -\Lambda^s Vdt +  \sum_{k=1}^{\infty}\left[\nabla_h\widetilde{p}_k+b_k\cdot\nabla V \right]\circ dW^k  ,  
        \\
        &\partial_z p  =0 ,\quad  \partial_z \widetilde{p}_k  =0, 
        \\
        &\nabla_h \cdot V + \partial_z w =0,
    \end{split} 
\end{align}
where $\Lambda^s=(-\Delta)^{s/2}$ denotes the fractional Laplacian with $s\in[1,2)$, $\nabla_h=(\partial_{x_1},\partial_{x_2})$ is the 2D horizontal gradient  and $\nabla=(\partial_{x_1},\partial_{x_2}, \partial_z)$ is the full 3D gradient. Here $V$ and $w$  represent the horizontal and vertical velocity component, respectively, $f_0\in \mathbb{R}$ is the Coriolis parameter, $p$ denotes the pressure, and $(\widetilde{p}_k)_{k\geq 1}$ are the components of the turbulent pressure \cite{mikulevicius2001equations}. $(W^k)_{k\geq 1}$ is a sequence of independent standard Brownian motions, and $(b_k)_{k\geq 1}$ are divergence-free vectors representing the coefficients of the transport noise. The system satisfies the following boundary conditions:  
\begin{equation}\label{BC-T3}
    V, w \text{ are periodic in }  (x_1,x_2,z) \text{ with period }  1,  \qquad V  \text{ is even in }  z  \text{ and }  w  \text{ is odd in }  z.
\end{equation}
Note that when the noise $b_k=(b_k^1, b_k^2, b_k^3)$ satisfies $(b_k^1,b_k^2)$ being even in $z$ and $b_k^3$ being odd in $z$, the domain consisting of periodic functions adhering to such symmetry conditions remains invariant under the dynamics of system \eqref{PE-system}.

The fully viscous PE (\(s=2\)) is derived from the Navier-Stokes equations \cite{azerad2001mathematical,li2019primitive} and is widely used in the study of large-scale oceanic and atmospheric dynamics. In the deterministic setting, the global existence and uniqueness of strong solutions have been demonstrated in \cite{cao2007global,kobelkov2006existence}. On the other hand, the inviscid PE ($s=0$) is ill-posed in Sobolev spaces and the Gevrey class of order strictly greater than 1 \cite{renardy2009ill,han2016ill,ibrahim2021finite}. While they are locally well-posed in the analytic class \cite{kukavica2011local,ghoul2022effect}, finite time blowup of solutions has been established \cite{cao2015finite,wong2015blowup,ibrahim2021finite,collot2023stable}.  In the stochastic setting, the well-posedness of the fully viscous PE has been investigated under the influence of multiplicative noise \cite{debussche2011local,debussche2012global} and transport noise \cite{brzezniak2021well,agresti2022stochastic,agresti2024stochastic,agresti2023primitive}.  The inviscid PE with multiplicative noise perturbation has been studied more recently in \cite{hu2023local,hu2023pathwise,hu2024regularization}. These studies collectively highlight the critical role of viscosity in determining the well-posedness/ill-posedness and global existence/finite time blowup of solutions to the PE system. Contrasting to the Navier-Stokes and Euler equations, such a property is unique to the PE system. This naturally motivates the study of PE systems with fractional dissipation, which interpolate between the fully viscous and inviscid cases, and the exploration of how the system's behavior evolves as the dissipation index varies.

Fractional dissipation introduces ``weaker'' and ``non-local'' dissipation compared to the classical dissipation arising from the full Laplacian. This poses significant challenges in mathematical analysis while offering intriguing opportunities for modeling physical phenomena. Fractional dissipation effectively captures anomalous diffusion, memory effects, and non-local interactions, making it a powerful tool in turbulence modeling \cite{caffarelli2009some,carrillo2017mathematical}. Its applications have expanded across fluid dynamics, appearing in studies of the Euler equations \cite{constantin2014unique,constantin2012nonlinear}, Boussinesq equations \cite{wu2014well,yang2014global}, quasi-geostrophic equations \cite{kiselev2007global,caffarelli2010drift,abdo2024dirichlet}, and magnetohydrodynamics \cite{wu20182d,dai2019class}. However, very little is known about the PE with fractional dissipation. To the best of our knowledge, the only result in this direction is an ongoing work \cite{abdo2024primitive} by the second author of this paper and his collaborators, which explores 2D fractional dissipative PE in the deterministic setting. On the other hand, nothing has been investigated for the 3D fractional dissipative PE, and the study of this system in the stochastic setting is completely open.

Transport noise, initially introduced in \cite{kraichnan1968small,kraichnan1994anomalous} to model small-scale turbulence effects, has also gained prominence in stochastic fluid mechanics. It has been extensively studied in the context of stochastic Navier-Stokes equations \cite{mikulevicius2001equations,mikulevicius2004stochastic,flandoli2022additive} and has more recently been applied to the fully viscous PE \cite{agresti2022stochastic,agresti2024stochastic,agresti2023primitive}. However, existing results do not extend to the fractionally dissipated PE due to the lack of full dissipation. In this work, we focus on the Stratonovich formulation of transport noise for several reasons. First, Stratonovich noise is more compatible with numerical simulations due to Wong-Zakai convergence results \cite{flandoli2011random} and two-scale type arguments \cite{flandoli2022additive,debussche2024second}. Second, energy estimates in Sobolev spaces are unattainable in the It\^o formulation because weak dissipation cannot sufficiently counteract energy input from the noise. Finally, recent studies have demonstrated the regularization effects of Stratonovich transport noise \cite{flandoli2021high,flandoli2021delayed,luo2023enhanced,agresti2024global}, suggesting that such noise may play a pivotal role in understanding the fractionally dissipated PE. These insights motivate our exploration of the system’s behavior under combined fractional dissipation and transport noise.

\medskip
\noindent\textbf{Challenges and key innovations.} We highlight some mathematical difficulties in the analysis of system \eqref{PE-system} and discuss important innovations. In contrast to the 2D case, where global existence of solutions in the subcritical regime is established \cite{abdo2024primitive}, the 3D setting lacks several favorable properties present in two dimensions. As a result, the techniques and methodologies employed in \cite{abdo2024primitive} cannot be directly applied, and global existence may not be expected in the 3D case. Instead, we focus on establishing local well-posedness.

A primary difficulty arises from the lack of full dissipation, requiring careful treatment of the second order derivatives from the It\^o-Stratonovich corrector.
A crucial step involves leveraging the cancellation between the It\^o-Stratonovich corrector and the energy input from the noise after applying It\^o’s formula. Specifically, we derive the key estimate:
\begin{align}\label{e.121101}
    \langle \LL^{\s}\Pc(b_k\cdot\nabla\Pc(b_k\cdot\nabla) V),\LL^{\s} V\rangle+\langle\LL^{\s}\Pc(b_k\cdot\nabla V),\LL^{\s}\Pc(b_k\cdot\nabla V)\rangle\lesssim_{b_k} \|V\|_{\s+\frac12}^2,
\end{align}
where $\Pc$ is the hydrostatic Leray projection defined in \eqref{e.062501}, and  $\|V\|_{\s}$ denotes the Sobolev $H^{\s}$ norm. 

In the context of the Euler equations \cite{crisan2019solution,lang2023well} and Boussinesq equations \cite{alonso2020well} in vorticity form, similar estimates to \eqref{e.121101}, though without $\Pc$, have been derived by exploiting commutator cancellations of pseudo-differential operators \cite{hormander2007analysis,taylor2017pseudodifferential}, resulting in bounds of $\|V\|_{\s}^2$. These results were further generalized in \cite{tang2022general}, encompassing the usual Leray projection and various fluid models. However, the hydrostatic Leray projection behaves less favorably than the usual Leray projection. Consequently, prior results cannot be directly applied. For the hydrostatic Leray projection, we establish the following commutator estimate (see Lemma \ref{l.072101}):
\begin{align*}
    \|[\Pc, b_k\cdot\nabla]f\|_{\s}\lesssim\|b_k\|_{\s+1}\|f\|_{\s}+\|\partial_zb_k^h\|_{\s-1}\|f\|_{\s+1}, 
\end{align*}
for sufficiently smooth $f$. Here the commutator of two operators $A$ and $B$ is denoted by $[A,B] := AB-BA$, and $b_k^h$ is the horizontal component of the vector field $b_k$. Notably, it is known that under the usual Leray projection a better bound $\|f\|_{\s}$ can be achieved instead of $\|f\|_{\s+1}$. This is essentially due to the fact that the symbol of the hydrostatic Leray projection is singular along the entire  $k_3$-axis in frequency space, whereas the usual Leray projection has only a single singularity at the origin. Therefore, proving \eqref{e.121101} requires rewriting the left-hand side into a suitable combination of commutators and carefully balancing the derivatives using commutator estimates in negative Sobolev norms (see Lemma \ref{l.062701}).

The bound in \eqref{e.121101}, which matches the order of the nonlinear term, can be controlled via interpolation with fractional dissipation in the subcritical case ($s\in(1,2)$). For the critical case (\(s = 1\)), we require both small initial conditions and small noise (through \(\partial_z b_k^h\)). In the supercritical case (\(s < 1\)), the dissipation is insufficient to control the highest order term due to the noise and the nonlinear term, necessitating analytic initial data, as in \cite{hu2023local}. This is consistent with the results in \cite{abdo2024primitive} where the ill-posedness in the supercritical case and in the critical case with large initial data is established. Furthermore, in the supercritical case, the cancellation observed in \eqref{e.121101} fails in the analytic setting unless \(b_k\) is independent of spatial variables, see remarks in Section~\ref{s.121102}. Thus, establishing well-posedness for the supercritical case with general transport noise remains an open problem.

To prove pathwise uniqueness, we apply a double cutoff technique to address difficulties arising from the nonlinear term. This approach can be applied to improve existing results, such as those in \cite{hu2023local}, demonstrating the existence of pathwise solutions in the space of analytic functions. 

\medskip
\noindent\textbf{Organization of the paper.} The rest of this paper is organized as follows. In Section~\ref{section:preliminary}, we introduce the mathematical preliminaries, set up the problem, and state the main result. Section~\ref{section:inviscid} is devoted to proving the main result. Specifically, we derive uniform estimates for the truncated cutoff system in Section~\ref{s.070701}, establish the existence of martingale solutions using standard compactness arguments in Section~\ref{s.121001}, and prove pathwise uniqueness, thereby completing the proof of the main result in Section~\ref{s.121101}. Some remarks on the supercritical case are provided in Section~\ref{s.121102}. Finally, auxiliary lemmas and technical commutator estimates are presented in Appendices~\ref{section:auxlemma} and \ref{s.121103}, respectively.

\section{Preliminaries and the main Result}\label{section:preliminary}

In this section, we introduce notations and assumptions and state our main result. The universal constant $C$ appearing in the paper may change from line to line. When necessary, we shall use subscripts to emphasize the dependence of $C$ on some parameters. 

\subsection{Functional settings} 

Let $\bx:= (\bx',z) = (x_1, x_2, z)\in \mathbb{T}^3$, where $\bx'$ and $z$ represent the horizontal and vertical variables, respectively, and $\mathbb{T}^3=\mathbb R^3/\mathbb Z^3$ denotes the three-dimensional torus with unit volume. Denote the $L^2$ norm of a function $f$ as
$$
    \|f \|:=\|f\|_{L^2(\mathbb{T}^3)} = \Big(\int_{\mathbb{T}^3} |f(\bx)|^2 d\bx\Big)^{\frac{1}{2}},
$$
associated with the inner product
$  \langle f,g\rangle = \int_{\mathbb{T}^3} f(\bx)g(\bx) d\bx
$
for $f,g \in L^2(\mathbb{T}^3)$. For a function $f \in L^2(\mathbb{T}^3)$, let $\widehat{f}_{\bk}$ denote its Fourier coefficient such that
\begin{equation*}
    f(\bx) = \sum\limits_{\bk\in 2\pi\mathbb{Z}^3} \widehat{f}_{\bk} e^{ i\bk\cdot \bx}, \qquad \widehat{f}_{\bk} = \int_{\mathbb{T}^3} e^{- i\bk\cdot \bx} f(\bx) d\bx.
\end{equation*}
For $\sigma \geq 0$, the Sobolev $H^{\s}$ norm and the $\dot{H}^{\sigma}$ semi-norm are defined as
\begin{eqnarray*}
&&\hskip-.1in
\|f \|_{H^{\s}}:= \Big(\sum\limits_{\bk\in 2\pi\mathbb{Z}^3} (1+|\bk|^{2\sigma}) |\widehat{f}_{\bk}|^2  \Big)^{\frac{1}{2}}, \qquad \|f \|_{\dot{H}^{\sigma}}:= \Big(\sum\limits_{\bk\in  2\pi\mathbb{Z}^3} |\bk|^{2\sigma} |\widehat{f}_{\bk}|^2  \Big)^{\frac{1}{2}}.
\end{eqnarray*}
Note that these two norms are equivalent for functions with zero mean. For convenience, we denote $H:=L^2$. When a function $f=f(x')$ depends only on the horizontal variables, we denote by $H^\sigma_h$ the corresponding Sobolev spaces. For further details on Sobolev spaces, we refer the reader to \cite{adams2003sobolev}.

The divergence-free condition from \eqref{PE-system} and the oddness of $w$ in $z$ imply that
\begin{equation*}
    \int_0^1 \nabla_h \cdot V (\bx',z)dz = 0.
\end{equation*}
Since $b_k$ are divergence-free, integrating the momentum equation in system \eqref{PE-system} over $\mathbb T^3$ shows that $V$ has zero mean for any $t>0$, provided the initial data $V_0$ has zero mean. Furthermore, as we consider $V$ to be even in the $z$ variable, it follows that $V\in \Hb$ where
\begin{equation*}
    \Hb := \left\{ f\in L^2(\mathbb{T}^3) : \int_0^1 \nabla_h \cdot f (\bx',z)dz = 0 ,\, f \text{ has zero mean and is even in } z\right\}.
\end{equation*}
We define $\Hb^{\sigma}= H^\s \cap \Hb$ for $\sigma>0$ and use the notation $\|\cdot\|_{\s}$ to denote the corresponding Sobolev norm.

The hydrostatic Leray projection $\mathcal P$ is defined as 
\begin{align}\label{e.062501}
    \mathcal P \varphi  = \varphi - \nabla_h  \Delta_h^{-1} \nabla_h \cdot \overline{\varphi}
\end{align}
for any regular velocity field $\varphi$, where $\bar{\varphi }(\bx') = \int_0^1 \varphi (\bx',z) dz$ represents the barotropic part of $\varphi$. 
We also define 
\[\Qb \varphi =(I-\Pc)\varphi=  \nabla_h  \Delta_h^{-1} \nabla_h \cdot \overline{\varphi}, \]
which is a function depending only on the horizontal variables $x'$. 

\subsection{Assumptions on the noise}\label{sec:stochasticprelim}
Let $\Sc = \left(\Omega, \Fc, \Fb, \Pb\right)$ be a stochastic basis with filtration $\Fb = \Fct$ satisfying the usual conditions.  Let $\Uc$ be a separable Hilbert space, and let $\Wb$ be an $\Fb$-adapted cylindrical Wiener process with reproducing kernel Hilbert space $\Uc$ on $\Sc$. Let $\lbrace e_k \rbrace_{k = 1}^\infty$ be an orthonormal basis of $\Uc$, then $\Wb$ may be formally written as $\Wb = \sum_{k=1}^\infty e_kW^k$, where $\{W^k\}_{k=1}^\infty$ are independent standard Brownian motions on $\Sc$. If we define the linear operators $P,\mathcal{B}:\Uc\to H$ by 
\[Pe_k=\widetilde{p}_k, \quad \mathcal{B}e_k = b_k, \quad k\geq 1,\]
then the noise term in \eqref{PE-system} is obtained from $\Uc$ as 
\begin{align*}
    \sum_{k=1}^{\infty}\left[\nabla_h\widetilde{p}_k+b_k\cdot\nabla V \right]\circ dW^k= [\nabla_hP+\mathcal{B}\cdot\nabla V]\circ d\Wb . 
\end{align*}
The pressure term $(\widetilde{p}_k)_{k\geq 1}$ will be eliminated after applying the hydrostatic Leray projector \eqref{e.062501}. We assume that each $b_k=(b_{k}^1,b_{k}^2,b_{k}^3)$ is divergence-free and has zero mean, with $b_k^3$ being odd in $z$ and $(b_{k}^1,b_{k}^2)$ being even in $z$. The regularity assumption on $b:=(b_k)_{k\geq1}$ is 
\begin{align}\label{e.121003}
    (b_k)_{k\geq1}\in \ell^2(\Nb,H^{\s+3}). 
\end{align}
In the case of $s=1$, we additionally assume that 
\begin{align}\label{e.121002}
    \sum_k\|b_k\|_{\s+3}\|\partial_zb_k^h\|_{\s-\frac32}\leq \frac{1}{C},
\end{align}
where $C$ is a universal constant from Sobolev inequalities.

\subsection{Notion of solutions}

In this section, we introduce the notions of  pathwise solutions (strong solutions in the stochastic sense) and martingale solutions (weak solutions in the stochastic sense) for system \eqref{PE-system}. For notational simplicity, we define:
\begin{equation*}
    Q(g,f) = g\cdot \nabla_h f - \left(\int_0^z (\nabla_h \cdot  g) (\widetilde z)  d\widetilde z \right) \partial_z f, \ \ F(g) :=  f_0 g^{\perp} , \ \ B_k g = b_k\cdot\nabla g. 
\end{equation*}
After applying the hydrostatic Leray projection \eqref{e.062501}, the pressure terms are eliminated, and system \eqref{PE-system} can be rewritten as:
\begin{align}\label{e.052502}
    \begin{split}
        &d V + \big[\mathcal{P}Q( V, V) + \mathcal{P}F(V) + \Lambda^s V\big]dt = \sum_{k=1}^\infty\mathcal{P}B_kV \circ dW^k,
    \\
    &V(0) = V_0. 
    \end{split}
\end{align}
The corresponding It\^o form (see, for example, \cite{agresti2022stochastic} the conversion from Stratonovich noise to It\^o noise) of the equation is: 
\begin{align}\label{e.052503}
    \begin{split}
        &d V + \big[\mathcal{P}Q( V, V) + \mathcal{P}F(V) + \Lambda^s V\big]dt = \frac12\sum_{k=1}^{\infty}(\mathcal{P}B_k)^2V dt +  \sum_{k=1}^\infty\mathcal{P}B_kV dW^k,
    \\
    &V(0) = V_0. 
    \end{split}
\end{align}

\begin{definition}[Pathwise solution]\label{definition:pathwise-solution}
    Let the initial condition $V_0 \in L^2\left(\Omega; \Hb^{\sigma}\right)$ be $\Fc_0$-measurable.
    \begin{enumerate}
	\item A pair $(V, \eta)$ is called a \emph{local pathwise solution} of system \eqref{e.052502} if $\eta$ is a strictly positive $\Fb$-stopping time and $V(\cdot \wedge \eta)$ is a progressively measurable stochastic process satisfying $\Pb$-a.s.,
    \begin{equation*}
        V\left( \cdot \wedge \eta \right) \in L^2\left(\Omega; C\left([0, \infty), \Hb^{\sigma}\right)\right), 
    \end{equation*}
    and for every $t\geq 0$, the following identity holds in $\Hb$:
    \begin{align*}
        \begin{split}
        &V\left(t \wedge \eta\right) + \int_0^{t \wedge \eta} \big[\Pc Q(V,V)+\Pc F(V)
        +\Lambda^sV\big] dr \\
        = & V(0) + \frac12\sum_{k=1}^{\infty}\int_0^{t \wedge \eta}(\mathcal{P}B_k)^2V dr + \sum_{k=1}^\infty\int_0^{t \wedge \eta} \mathcal{P}B_kV dW_r^k.
        \end{split}
    \end{align*}
 
   \item A triple $\left(V, (\eta_n)_{n\geq 1},\xi\right)$ is called a \emph{maximal pathwise solution} if each pair $(V,\eta_n)$ is a local pathwise solution, $\eta_n$ is an increasing sequence of stopping times with $\lim\limits_{n\to \infty} \eta_n = \xi$  almost surely, and
   \begin{align*}
    \sup\limits_{t\in[0,\eta_n]} \|V\|_{\s} \geq n,  \quad \text{ on the set } \{\xi<\infty\}.
   \end{align*}
\end{enumerate}
\end{definition}

\begin{definition}[Martingale solution]\label{definition:martingale-solution}
    Let $T>0$ and $V_0 \in L^2\left(\Omega; \Hb^{\sigma}\right)$ be an $\Fc_0$-measurable random variable. A martingale solution to equation \eqref{e.052502} on $[0,T]$ is a quadruple $(\widetilde{V}_0,\widetilde{V},\widetilde{\Wb},\widetilde{\Sc})$ such that 
    \begin{enumerate}
        \item  $\widetilde{\Sc}=(\widetilde{\Omega},\widetilde{\Fc},\widetilde{\Fb},\widetilde{\Pb})$ is a stochastic basis,  over which $\widetilde{\Wb}$ is an $\widetilde{\Fb}$-adapted cylindrical Brownian motion with components $(\widetilde{W}^k)_{k\geq 0}$;
        
        \item $\widetilde{V}_0 \in L^2(\widetilde{\Omega},\Hb^{\sigma})$ has the same law as $V_0$;
        \item $\widetilde V$ is a progressively measurable process such that  
    \begin{equation*}
        \widetilde{V} \in L^2\left(\widetilde{\Omega}; C\left([0, T];\Hb^{\sigma}\right)\right), 
    \end{equation*}
    $\widetilde{\Pb}$-a.s., and for every $t\in[0,T]$, the following identity holds in $\Hb$:
    \begin{align*}
        \begin{split}
        &\widetilde{V}\left(t\right) + \int_0^{t} \big[\Pc Q(\widetilde{V},\widetilde{V})+\Pc F(\widetilde{V})
        +\Lambda^s\widetilde{V}\big] dr \\
        = &\widetilde{V}(0) + \frac12\sum_{k=1}^{\infty}\int_0^{t}(\mathcal{P}B_k)^2\widetilde{V} dr + \sum_{k=1}^\infty\int_0^{t} \mathcal{P}B_k\widetilde{V} d\widetilde{W}_r^k.
        \end{split}
    \end{align*}
        \end{enumerate}
\end{definition}

\subsection{Main result} The main result of this paper is stated in the following theorem.
\begin{theorem}\label{t.062501}
    Let $\sigma>3$, $s\in(1,2)$ and assume the noise coefficient satisfies \eqref{e.121003}. Then, for any $V_0\in L^2(\Omega,\Hb^{\sigma})$, there exists a maximal pathwise solution $(V,(\eta_n)_{n\geq1},\xi)$ to system \eqref{e.052502}. 
    
    The same result holds for the case $s=1$ if we assume additionally that the noise satisfies \eqref{e.121002} and that $\|V_0\|_{\s}< \frac{1}{C_0}$, where $C_0$ is a universal constant from Sobolev embeddings. 
\end{theorem}

\section{Proof of the main result}\label{section:inviscid}
This section is dedicated to proving Theorem \ref{t.062501}. We begin by deriving uniform estimates for the truncated cutoff system in Section~\ref{s.070701}. Next, we establish the existence of a martingale solution to the cutoff system in Section~\ref{s.121001} using standard compactness arguments. Finally, Theorem~\ref{t.062501} is proved at the end of Section~\ref{p.121001} after we demonstrate pathwise uniqueness. 

As mentioned in the introduction, a major difficulty arises from the singularity of the hydrostatic Leray projector and the fractional dissipation. In particular, obtaining the uniform estimates in Section~\ref{s.070701} requires carefully controlling commutator estimates involving the hydrostatic Leray projector. In addition, the usual cutoff scheme is inadequate for proving pathwise uniqueness in Section~\ref{p.121001} due to the nonlinear terms. To address this, a double cutoff scheme is introduced to overcome the associated challenges. Throughout this section, 
we assume that the noise coefficients satisfy the conditions outlined in Section~\ref{sec:stochasticprelim}. 

Let $\theta_\rho(x)\in C^\infty(\mathbb R)$ be a non-increasing cutoff function defined as 
\begin{equation}\label{eqn:rho}
    \theta_\rho(x)= \begin{cases}1, & \text { if }|x|\leq \frac{\rho}{2}, \\ 0 & \text { if }|x| >\rho.\end{cases}
\end{equation} 
Consider the cutoff system associated with the equation \eqref{e.052503}, 
\begin{align}\label{e.052504}
    \begin{split}
        &d V + \big[\theta_\rho^2\mathcal{P}Q( V, V) + \mathcal{P}F(V) + \Lambda^s V\big]dt = \frac12\sum_{k=1}^{\infty}(\mathcal{P}B_k)^2V +  \sum_{k=1}^\infty\mathcal{P}B_kV dW^k,
    \\
    &V(0) = V_0, 
    \end{split}
\end{align}
where we denote by $\theta_\rho=\theta_\rho(\|V\|_{W^{1,\infty}})$ for convenience. 

\subsection{Analysis of the Galerkin system}\label{s.070701}
In this subsection, we derive a uniform energy estimate for the Galerkin system associated with the cutoff system \eqref{e.052504}. For $\bk\in 2\pi\mathbb{Z}^3$, define
\begin{equation*}
\phi_{\bk} = \phi_{k_1,k_2,k_3} :=
\begin{cases}
\sqrt{2}e^{ i\left( k_1 x_1 + k_2 x_2 \right)}\cos( k_3 z) & \text{if} \;  k_3\neq 0,\\
e^{ i\left( k_1 x_1 + k_2 x_2 \right)} & \text{if} \;  k_3=0,  \label{phik}
\end{cases}
\end{equation*}
and
\begin{eqnarray*}
	&&\hskip-.28in
	\mathcal{E}:=  \left\{ \phi \in L^2(\mathbb{T}^3) \; | \; \phi= \sum\limits_{\bk\in 2\pi\mathbb{Z}^3} a_{\bk} \phi_{\bk}, \; a_{-k_1, -k_2,k_3}=a_{k_1,k_2,k_3}^{*}, \; \sum\limits_{\bk\in 2\pi\mathbb{Z}^3} |a_{\bk}|^2 < \infty \right\},
\end{eqnarray*}
where $a^*$ denotes the complex conjugate of $a$. Notice that $\mathcal{E}$ is a closed subspace of $L^2(\mathbb{T}^3)$, consisting of real-valued functions that are even in the $z$ variable. 

For any $n\in \N$, define the finite-dimensional subspace
\begin{eqnarray*}
	&&\hskip-.28in
	\mathcal{E}_n:=  \left\{ \phi \in L^2(\mathbb{T}^3)\; | \; \phi= \sum\limits_{|\bk|\leq n} a_{\bk} \phi_{\bk}, \; a_{-k_1, -k_2, k_3}=a_{k_1,k_2, k_3}^{*}  \right\},
\end{eqnarray*}
For any function $f\in L^2(\mathbb{T}^3)$, denote its Fourier coefficients by 
\begin{equation*}
f_{\bk} :=\int_{\mathbb{T}^3} f(\bx)\phi^*_{\bk}(\bx) d\bx,
\end{equation*}
and define the projection $P_n f := \sum_{|\bk|\leq n} f_{\bk} \phi_{\bk}$ for $n\in \N$. Then, $P_n$ is an orthogonal projection from $L^2(\mathbb{T}^3)$ to $\mathcal{E}_n$. 

For $n\in\mathbb N$, the Galerkin approximation of system \eqref{e.052504} at order $n$ is given by
\begin{align}\label{e.052505}
    \begin{split}
        &d V_n + \big[\theta_\rho^2P_n\mathcal{P}Q( V_n, V_n) + P_n\mathcal{P}F(V_n) + \Lambda^s V_n\big]dt = \frac12\sum_{k=1}^{\infty}P_n(\mathcal{P}B_k)^2V_n dt +  \sum_{k=1}^\infty P_n\mathcal{P}B_kV_n dW^k,
    \\
    &V_n(0) = P_nV_0, 
    \end{split}
\end{align}
where $\theta_\rho=\theta_\rho(\|V_n\|_{W^{1,\infty}})$. 
Since the coefficients are locally Lipschitz, the Galerkin system has a unique local solution. As the cancellation of the nonlinear term holds for the Galerkin system, the solution is indeed global. 

We now establish the following uniform energy estimate. 

\begin{proposition}\label{t.062502}
    Let $p\geq 2$, $T\geq 0$, $s\in[1,2)$, $\sigma>3$, and let $V_0\in L^p(\Omega,\Hb^{\sigma})$ be $\Fc_0$-measurable. In the case $s=1$, we additionally assume that the cutoff parameter satisfies $\rho<\frac{1}{2C_{\s}}$, where $C_{\s}$ is a constant arising from the Sobolev embedding. Let $V_n$ be the solution to the Galerkin system \eqref{e.052505}. Then for some universal constant $C$ independent of $n$, the following results hold:
    \begin{itemize}
        \item[(1)]Uniform energy bound:   
        \[\mathbb E \Big[\sup\limits_{r\in[0,T]} \|V_n(r)\|_{\s}^p +  \int_0^T \|V_n(r)\|_{\s}^{p-2}\|V_n(r)\|_{\sigma+\frac{s}{2}}^2 dr \Big]  \leq C(1+ \mathbb E \|V_0\|_{\s}^p).\]
        \item[(2)]For any $\alpha\in[0,\frac12)$, 
        \[ \Eb \left|\int_0^{\cdot}\sum_{k}P_n\Pc B_kV_ndW_r^k\right|_{W^{\alpha,p}(0,T;\Hb^{\s-1})}^p\leq C(1+ \mathbb E \|V_0\|_{\s}^p).\]
        \item[(3)] The following bound holds: 
        \[ \Eb\Big\|V_n - \int_0^{\cdot}\sum_kP_n\Pc B_k V_ndW_r^k\Big\|_{W^{1,2}(0,T;\Hb^{\s-2+\frac{s}{2}})}\leq C(1+ \mathbb E \|V_0\|_{\s}^2).\]
    \end{itemize}
\end{proposition}
\begin{proof}
     Applying It\^o's formula to the functional $\|\LL^{\s}V_n\|^p$, by \eqref{e.052505} we have 
    \begin{align*}
        \begin{split}
            d\|\LL^{\s}V_n\|^{p} &= -p\langle \theta_\rho^2P_n\Pc Q(V_n,V_n)+P_n\Pc F(V_n)+\Lambda^s V_n,\Lambda^{2\sigma}V_n\rangle\|\LL^{\s}V_n\|^{p-2}dt \\
            &+\frac{p}{2}\sum_{k}\left(\langle P_n(\mathcal{P}B_k)^2V_n,\Lambda^{2\sigma}V_n \rangle+\|\LL^{\s}P_n\Pc B_kV_n\|^2\right)\|\LL^{\s}V_n\|^{p-2}dt\\
            &+\frac{p(p-2)}{2}\|\LL^{\s}V_n\|^{p-4}\sum_k\langle \LL^{2\s}V_n,P_n\Pc B_k V_n\rangle^2 dt+p\|\LL^{\s}V_n\|^{p-2}\sum_{k}\langle P_n\Pc B_k V_n, \LL^{2\sigma} V_n\rangle dW^k\\
            & = I_0 dt +I_1 dt +I_2 dt +I_3 d\Wb. 
        \end{split}
    \end{align*}

      We first consider 
    \begin{align*}
        \int_0^t I_1 dr = \int_0^t\frac{p}{2}\sum_{k}\left(\langle P_n(\mathcal{P}B_k)^2V_n,\Lambda^{2\s}V_n \rangle+\|\LL^{\s}P_n\Pc B_kV_n\|^2\right)\|\LL^{\s}V_n\|^{p-2}dr.
    \end{align*}
    Note that the self-adjoint operator $\Pc$ commutes with $\LL$ in the periodic case, $\Pc V_n = V_n$, and each $b_k$ in $B_k=b_k\cdot\nabla$ is divergence-free. Additionally, $P_n$ commutes with $\LL$ and satisfies $\|P_n\cdot\|\leq \|\cdot\|$. Therefore, one has 
    \begin{align}\label{e.070701}
        \begin{split}
            &\hspace{.5cm}\langle P_n(\mathcal{P}B_k)^2V_n,\Lambda^{2\s}V_n \rangle+\|\LL^{\s}P_n\Pc B_kV_n\|^2\\
            &\leq \langle  B_k \Pc B_k V_n,\LL^{2\s}V_n \rangle+ \langle \LL^{\s}B_k V_n, \LL^{\s}\Pc B_k V_n\rangle\\
            & = \langle  B_k [\Pc,  B_k] V_n,\LL^{2\s}V_n \rangle+ \langle  B_k B_k V_n,\LL^{2\s}V_n \rangle+ \langle \LL^{\s}B_k V_n, \LL^{\s}[\Pc,  B_k] V_n\rangle+\langle \LL^{\s}B_k V_n, \LL^{\s} B_k V_n\rangle\\
            & = \langle [\LL^{\s},B_k ][\Pc,  B_k] V_n,\LL^{\s}V_n \rangle + \langle \LL^{\s}[\Pc,  B_k] V_n,[\LL^{\s},B_k ] V_n \rangle\\
            &\qquad+ \langle  B_k B_k V_n,\LL^{2\s}V_n \rangle+\langle \LL^{\s}B_k V_n, \LL^{\s} B_k V_n\rangle\\
            &= \langle [\LL^{\s},B_k ][\Pc,  B_k] V_n,\LL^{\s}V_n \rangle + \langle \LL^{\s}[\Pc,  B_k] V_n,[\LL^{\s},B_k ] V_n \rangle\\
            &\qquad+ \langle  [\LL^{\s},B_k ] B_k V_n,\LL^{\s}V_n \rangle + \langle [\LL^{\s},B_k ] V_n, \LL^{\s} B_k V_n\rangle\\
            &= \langle [\LL^{\s},B_k ][\Pc,  B_k] V_n,\LL^{\s}V_n \rangle + \langle \LL^{\s}[\Pc,  B_k] V_n,[\LL^{\s},B_k ] V_n \rangle\\
            &\qquad+ \langle  [[\LL^{\s},B_k ],B_k]V_n,\LL^{\s}V_n \rangle + \langle [\LL^{\s},B_k ] V_n, [\LL^{\s},B_k ] V_n\rangle.\\
        \end{split}
    \end{align}
Since $\s>\frac52$, it follows from Lemma \ref{l.071701} and the Sobolev embedding that
\begin{align}\label{e.072102}
    \begin{split}
        \|[\LL^{\s},B_k ] V_n\|\leq \sum_j\|[\LL^{\s},b_k^j](\partial_jV_n)\|
        &\lesssim \|\nabla b_k\|_{L^{\infty}}\|\nabla V_n\|_{\s-1}+\|\LL^{\s} b_k\|\|\nabla V_n\|_{L^{\infty}}\\
        \lesssim \|b_k\|_{\s}\|V_n\|_{\s},
    \end{split}
\end{align}
which leads to 
\begin{align*}
    \langle [\LL^{\s},B_k ] V_n, [\LL^{\s},B_k ] V_n\rangle\lesssim \|b_k\|_{\s}^2\|V_n\|_{\s}^2. 
\end{align*}
Thanks to Lemma~\ref{l.072001}, we also have 
\begin{align*}
    \langle  [[\LL^{\s},B_k ],B_k]V_n,\LL^{\s}V_n \rangle\leq \|[[\LL^{\s},B_k ],B_k]V_n\|\|\LL^{\s}V_n \|\lesssim \|b_k\|_{\s+1}^2\|V_n\|_{\s}^2. 
\end{align*}
Applying Lemmas \ref{l.072102} and \ref{l.072101} we obtain
\begin{align*}
    \begin{split}
        \langle \LL^{\s}[\Pc,  B_k] V_n,[\LL^{\s},B_k ] V_n \rangle&\leq \|\LL^{\s-\frac12}[\Pc,  B_k] V_n\|\|\LL^{\frac12}[\LL^{\s},B_k ] V_n \|\\
        &\lesssim \left(\|b_k\|_{\s+\frac12}\|V_n\|_{\s-\frac12}+\|\partial_zb_k^h\|_{\s-\frac32}\|V_n\|_{\s+\frac12}\right)\left(\|b_k\|_{\s}\|V_n\|_{\s+\frac12}+\|b_k\|_{\s+3}\|V_n\|_{\s}\right)\\
        &\lesssim \|\partial_zb_k^h\|_{\s-\frac32}\|b_k\|_{\s}\|V_n\|_{\s+\frac12}^2+\|b_k\|_{\s+3}^2\|V_n\|_{\s}\|V_n\|_{\s+\frac12}+\|b_k\|_{\s+3}^2\|V_n\|_{\s}^2, 
    \end{split}
\end{align*}
where $b_k^h=(b_k^1,b_k^2)$. By Lemma \ref{l.072101}, we further have 
\begin{align*}
    \langle [\LL^{\s},B_k ][\Pc,  B_k] V_n,\LL^{\s}V_n \rangle &\leq \|\LL^{-\frac12}[\LL^{\s},B_k ][\Pc,  B_k] V_n\|\|\LL^{\s+\frac12}V_n\|\\
    &\lesssim \left(\|b_k\|_{\s+3}^2\|\|V_n\|_{\s-\frac12}+\|b_k\|_{\s+3}\|\partial_zb_k^h\|_{\s-\frac32}\|V_n\|_{\s+\frac12}\right)\|V_n\|_{\s+\frac12}\\
    &\lesssim \|b_k\|_{\s+3}^2\|\|V_n\|_{\s}\|V_n\|_{\s+\frac12} + \|b_k\|_{\s+3}\|\partial_zb_k^h\|_{\s-\frac32}\|V_n\|_{\s+\frac12}^2. 
\end{align*}
Combining the estimates above, it follows that
\begin{align}\label{e.121004}
    \begin{split}
        &\langle P_n(\mathcal{P}B_k)^2V_n,\Lambda^{2\s}V_n \rangle+\|\LL^{\s}P_n\Pc B_kV_n\|^2\\
        &\qquad \lesssim \|b_k\|_{\s+3}\|\partial_zb_k^h\|_{\s-\frac32}\|V_n\|_{\s+\frac12}^2+\|b_k\|_{\s+3}^2\|V_n\|_{\s}\|V_n\|_{\s+\frac12}+\|b_k\|_{\s+3}^2\|V_n\|_{\s}^2. 
    \end{split}
\end{align}
For $s>1$, by interpolation and Young's inequalities, and noting that $V_n$ has zero mean, one has $$\|V_n\|_{\s+\frac12}^2\lesssim\|\LL^{\s+\frac12}V_n\|^2\leq \varepsilon\|\LL^{\s+\frac{s}{2}}V_n\|^2+C_{\varepsilon}\|\LL^{\s}V_n\|^2.$$ 
Consequently,
\begin{align}
    \sum_k\left(\langle P_n(\mathcal{P}B_k)^2V_n,\Lambda^{2\s}V_n \rangle+\|\LL^{\s}P_n\Pc B_kV_n\|^2\right)
    \leq \frac14\|\LL^{\s+\frac{s}{2}}V_n\|^2+C\|b\|_{\ell^2(\Nb,H^{\s+3})}^2\|\LL^{\s}V_n\|^2. \label{s>1}
\end{align}
For $s=1$, we impose the assumption that
\[\sum_k\|b_k\|_{\s+3}\|\partial_zb_k^h\|_{\s-\frac32}\leq \frac{1}{5C},\]
where $C$ is a universal constant. Note that this assumption is automatically satisfied when $\partial_z b_k^h=0$, i.e., $b_k^h$ is independent of the $z$ variable. Under this assumption, from \eqref{e.121004} we have
\begin{align}
    &\sum_k\left(\langle P_n(\mathcal{P}B_k)^2V_n,\Lambda^{2\s}V_n \rangle+\|\LL^{\s}P_n\Pc B_kV_n\|^2\right)\nonumber\\
    &\qquad \leq C \|\LL^{\s+\frac12}V_n\|^2\sum_k\|b_k\|_{\s+3}\|\partial_zb_k^h\|_{\s-\frac32} + \|b\|_{\ell^2(\Nb,H^{\s+3})}^2\left(\|V_n\|_{\s}\|V_n\|_{\s+\frac12}+\|V_n\|_{\s}^2\right)\nonumber\\
    &\leq \frac14\|\LL^{\s+\frac12}V_n\|^2+C\left(1+\|b\|_{\ell^2(\Nb,H^{\s+3})}^4\right)\|\LL^{\s}V_n\|^2. \label{s=1}
\end{align}
Thus, under the given condition, for $1\leq s<2$, we derive
\begin{align}\label{e.072103}
    \int_0^t I_1 dr \leq \frac14 \int_0^t\|\LL^{\s+\frac{s}{2}}V_n\|^2\|\LL^{\s}V_n\|^{p-2}dr + C_b\int_0^t\|V_n\|_{\s}^pdr.
\end{align}

To estimate $I_2$ and $I_3$, note that since each $b_k$ is divergence-free,  we have
\begin{align*}
   \langle \LL^{2\s}V_n,P_n\Pc B_k V_n\rangle
    &= \langle \LL^{\s}V_n, \LL^{\s} B_k  V_n\rangle = \langle \LL^{\s}V_n, \LL^{\s} B_k  V_n\rangle -\langle \LL^{\s}V_n,  B_k \LL^{\s}V_n\rangle\\
    & = \langle \LL^{\s}V_n, [\LL^{\s}, B_k] V_n\rangle\leq \|b_k\|_{\s+2}\|V_n\|_{\s}^2,
\end{align*}
where the last inequality follows from estimate \eqref{e.072102}. Consequently, we obtain 
\begin{align*}
    \begin{split}
        \int_0^t I_2 dr =\frac{p(p-2)}{2} \int_0^t \|\LL^{\s}V_n\|^{p-4}\sum_k\langle \LL^{2\s}V_n,P_n\Pc B_k V_n\rangle^2dr\leq C_b\int_0^t  \|V_n\|_{\sigma}^p dr.
    \end{split}
\end{align*}
Using the Burkholder-Davis-Gundy inequality, we then estimate $I_3$:
\begin{align*}
    \begin{split}
        \Eb\sup_{r\in[0,t]}\left|\int_0^r I_3 d\Wb\right| &= p\Eb\sup_{r\in[0,t]}\left|\int_0^r \|\LL^{\s}V_n\|^{p-2}\sum_{k}\langle P_n\Pc B_k V_n, \LL^{2\sigma} V_n\rangle dW^k \right|\\
        &\leq C_p \Eb\left(\int_0^t\|V_n\|_{\s}^{2(p-2)}\sum_k\langle  B_k V_n, \LL^{2\sigma} V_n\rangle^2dr\right)^{\frac12}\\
        &= C_p \Eb\left(\int_0^t\|V_n\|_{\s}^{2(p-2)}\sum_k\langle  [\LL^{\s},B_k] V_n, \LL^{\sigma} V_n\rangle^2dr\right)^{\frac12}\\
        &\leq C_b\Eb\left(\int_0^t\|V_n\|_{\s}^{2p}dr\right)^{\frac12}\\
        &\leq \frac14\Eb\sup_{r\in[0,t]}\|V_n\|_{\s}^{p} + 4C_b^2 \Eb\int_0^t  \|V_n(r)\|_{\sigma}^p dr. 
    \end{split}
\end{align*}
Note that the constant $C_b$ above depends on $b$ through $\|b\|_{\ell^2(\Nb, H^{\s+2})}$. 

To estimate $I_0$, we first use Lemma \ref{l.071701} and the cutoff function $\theta$ to obtain 
\begin{align}\label{e.071701}
    \begin{split}
        &-p\int_0^t\theta_\rho^2\langle P_n\Pc Q(V_n,V_n), \LL^{2\s}V_n\rangle \|\LL^{\s}V_n\|^{p-2}dr
        \\
        =&-p\int_0^t\theta_\rho^2\langle P_n\Pc \LL^{\s-\frac12} Q(V_n,V_n), \LL^{\s+\frac12}V_n\rangle \|\LL^{\s}V_n\|^{p-2}dr\\
        \leq& pC_{\s}\int_0^t\theta_\rho^2\|V_n\|_{W^{1,\infty}}\|\LL^{\s+\frac12}V_n\|^2\|\LL^{\s}V_n\|^{p-2}dr\\
        \leq& p C_{\s} \rho\int_0^t\|\LL^{\s+\frac12}V_n\|^2\|\LL^{\s}V_n\|^{p-2}dr. 
    \end{split}
\end{align}
Besides, we have
$
    \langle P_n\Pc F(V_n), \LL^{2\s}V_n\rangle  = 0
$
since $F(V_n)$ is orthogonal to $V_n$. 
Therefore, when $s=1$ one has 
\begin{align*}
    \begin{split}
        \int_0^t I_0 dr\leq -p\int_0^t(1-C_{\sigma}\rho)\|\LL^{\s+\frac12}V_n\|^2\|\LL^{\s}V_n\|^{p-2}dr. 
    \end{split}
\end{align*}
For $1<s<2$,  applying interpolation and Young's inequalities, we deduce  $$\|\LL^{\s+\frac12}V_n\|^2\leq \varepsilon\|\LL^{\s}V_n\|^2+C_\varepsilon\|\LL^{\s+\frac{s}{2}}V_n\|^2.$$ 
Then we can bound \eqref{e.071701} to obtain 
\begin{align*}
    \begin{split}
        \int_0^t I_0dr &\leq pC_{\sigma}\rho\int_0^t\|\LL^{\s+\frac{1}{2}}V_n\|^2\|\LL^{\s}V_n\|^{p-2}dr - p\int_0^t\|\LL^{\s+\frac{s}{2}}V_n\|^2\|\LL^{\s}V_n\|^{p-2}dr\\
        &\leq - \frac p2\int_0^t\|\LL^{\s+\frac{s}{2}}V_n\|^2\|\LL^{\s}V_n\|^{p-2}dr+C_{p,\sigma,\rho}\int_0^t\|V_n\|_{\s}^{p}dr.
    \end{split}
\end{align*}

Summarizing the above estimates for $I_0$ to $I_3$, we obtain for any $t\in[0,T]$:
\begin{align*}
    \Eb\sup_{r\in[0,t]}\|\LL^{\s}V_n(r)\|^p +\Eb \int_0^t\|\LL^{\s+\frac{s}{2}}V_n\|^2\|\LL^{\s}V_n\|^{p-2}dr\leq C\left(\Eb\|\LL^{\s}V_n(0)\|^p + \Eb\int_0^t\|V_n\|_{\s}^pdr\right). 
\end{align*}
Since $V_n$ has zero mean, this leads to 
\begin{align*}
    \Eb\sup_{r\in[0,t]}\|V_n(r)\|_{\s}^p +\Eb \int_0^t\|V_n\|_{\s+\frac{s}{2}}^2\|V_n\|_{\s}^{p-2}dr\leq C\left(\Eb\|V_n(0)\|_{\s}^p + \Eb\int_0^t\|V_n\|_{\s}^pdr\right). 
\end{align*}
Applying Gr\"onwall's inequality, we conclude
\begin{align*}
    \begin{split}
        \Eb\left[\sup_{r\in[0,T]}\|V_n\|_{\s}^p+ \int_0^T\|V_n\|_{\s+\frac{s}{2}}^2\|V_n\|_{\s}^{p-2}dr\right]\leq C(1+\Eb\|V_0\|_{\s}^p). 
    \end{split}
\end{align*}
This proves (1). 

Next, by the Burkholder-Davis-Gundy inequality, for $\alpha\in [0,\frac12)$ and $p\geq 2$, one has 
\begin{align*}
    \begin{split}
        \Eb \left|\int_0^{\cdot}\sum_{k}P_n\Pc B_kV_ndW_r^k\right|_{W^{\alpha,p}(0,T;\Hb^{\s-1})}^p
        &\leq C_p\Eb\int_0^T\left(\sum_{k}\|B_kV_n\|_{\s-1}^2\right)^{\frac{p}{2}}dr\\
        &\leq \Eb C_p\int_0^T\|V_n\|_{\s}^p\left(\sum_{k}\|b_k\|_{\s-1}^2\right)^{\frac{p}{2}}dr\\
        &\leq C_{T,b,p} \Eb\sup_{r\in[0,T]}\|V_n\|_{\s}^p.
    \end{split}
\end{align*}
Combining this result with (1), we conclude (2).

Finally, in view of \eqref{e.052505}, we have 
\begin{align*}
    \begin{split}
        V_n(t) &- \int_0^t\sum_kP_n\Pc B_k V_ndW_r^k  \\
        &= P_nV_0 - \int_0^t\big[\theta_\rho^2P_n\mathcal{P}Q( V_n, V_n) + P_n\mathcal{P}F(V_n) + \Lambda^s V_n\big]dr + \frac12 \int_0^t \sum_kP_n(\Pc B_k)^2V_n dr.  
    \end{split}
\end{align*}
Since $\s>3$ and $s\in[1,2)$, one has $\s-2+\frac{s}{2}>\frac32$. Hence 
\begin{align*}
    \|\theta_\rho^2P_n\mathcal{P}Q( V_n, V_n) + P_n\mathcal{P}F(V_n) + \Lambda^s V_n\|_{\s-2+\frac{s}{2}}\leq C(1+\|V_n\|_{\s+\frac{s}{2}}^2),
\end{align*}
and 
\begin{align*}
    \|P_n(\Pc B_k)^2V_n\|_{\s-2+\frac{s}{2}}\leq C\|b_k\|_{\s-2+\frac{s}{2}}\|\Pc B_kV_n\|_{\s-1+\frac{s}{2}}\leq C\|b_k\|_{\s-1+\frac{s}{2}}^2\|V_n\|_{\s+\frac{s}{2}}. 
\end{align*}
We thus obtain 
\begin{align*}
    \begin{split}
        \Eb\Big\|V_n(t) &- \int_0^t\sum_kP_n\Pc B_k V_ndW_r^k\Big\|_{W^{1,2}(0,T;\Hb^{\s-2+\frac{s}{2}})}\\
        &\leq C\Eb\left[1+\|V_0\|_{\s}^2+\sup_{r\in[0,T]}\|V_n\|_{\s}^2+\int_0^T\|V_n\|_{\s+\frac{s}{2}}^2dr\right],
    \end{split}
\end{align*}
which implies (3) by utilizing (1) with $p=2$.

\end{proof}
\subsection{Compactness and martingale solutions}\label{s.121001}
In this subsection, we establish the existence of martingale solutions to system \eqref{e.052504}, leveraging the energy estimates obtained in Proposition~\ref{t.062502}. Before proceeding with the proof, we outline the necessary settings.

Recall that we have fixed a stochastic basis $\mathcal{S}=(\Omega, \mathcal{F}, \mathbb{F}, \mathbb{P})$. Define the path space 
\[\mathcal{X}= \Hb^{\s}\times L^2\left(0, T;  \Hb^{\s}\right) \cap C\left(\left[0, T\right]; \Hb^{\s-2}\right)\times C\left(\left[0, T\right]; \Uc\right). \]
Given any random initial data $V_0\in L^2(\Omega,\Hb^{\s})$, we let $\mu_0^n$ be the law of $P_nV_0$, and $\mu_V^n$ be the law of the corresponding solution $V_n$ to the approximating system \eqref{e.052505} with initial data $V_n(0) = P_nV_0 $, and also $\mu_{\mathbb{W}}$ the law of the cylindrical Wiener process on $C\left(\left[0, T\right]; \Uc\right)$. Define  $\mu^n =\mu_0^n\otimes \mu_V^n\otimes \mu_{\mathbb{W}} $ as their joint law in the path space $\mathcal{X}$. 

\begin{proposition}\label{p.070601}
    Assume the same conditions as in Proposition~\ref{t.062502} and  let $p>2$. Then there exists a stochastic basis $(\widetilde{\Omega},\widetilde{\mathcal{F}}, \widetilde{\mathbb{F}}, \widetilde{\mathbb{P}})$ and an $\mathcal{X}$-valued random variable $(\widetilde{V}_0,\widetilde{V},\widetilde{\mathbb W})$ over $\widetilde{\Omega}$ such that
    $\widetilde{V}$, adapted to the filtration $\widetilde{\mathbb{F}}$,
    is a solution to \eqref{e.052504} on $[0,T]$, with driving noise $\widetilde{\mathbb W}$ and initial data $\widetilde{V}_0$ having the same distribution as $V_0$. In short, $(\widetilde{V}_0,\widetilde{V},\widetilde{\Wb})$ over the stochastic basis $(\widetilde{\Omega},\widetilde{\mathcal{F}}, \widetilde{\mathbb{F}}, \widetilde{\mathbb{P}})$ is a martingale solution to \eqref{e.052504} on $[0,T]$.  Moreover, $\widetilde{V}$ satisfies 
    \begin{equation}
        \label{eq:martingale.solution.approx.regularity}
        \widetilde{V} \in L^2\left( \widetilde \Omega; C\left( [0, T]; \Hb^{\s}\right)\cap L^2\left( 0, T; \Hb^{\s+\frac{s}{2}} \right) \right).
    \end{equation}
\end{proposition}

\begin{proof}
    The proof is divided into two main steps. First, we establish the compactness of $\mu^n$ to obtain a limit as a random variable with values in $\mathcal{X}$ using the Skorokhod theorem. Then, we show that the limit gives a martingale solution to \eqref{e.052504}.

    \medskip

    \noindent\textbf{Step 1: Compactness.} By Lemma~\ref{lemma:aubin-lions}, the embedding
    \[L^2(0,T;\Hb^{\s+\frac{s}{2}})\cap W^{\frac14,2}(0,T;\Hb^{\s-1})\subset L^2(0,T;\Hb^{\s})\]
    is compact. In addition, by choosing $\alpha\in (0,1/2)$ such that $\alpha p>1$, both spaces $W^{1,2}(0,T;\Hb^{\s-2+\frac{s}{2}})$ and $W^{\alpha,p}(0,T;\Hb^{\s-2+\frac{s}{2}})$ are compactly embedded in $C([0,T];\Hb^{\s-2})$. These compact embeddings imply the tightness of the sequence $\lbrace \mu^n \rbrace_{n\geq 1}$ over $\mathcal{X}$. Indeed, if we let 
    \[B_R^1 = \left\{V\in L^2(0,T;\Hb^{\s+\frac{s}{2}})\cap W^{\frac14,2}(0,T;\Hb^{\s-1}): \|V\|_{L^2(0,T;\Hb^{\s+\frac{s}{2}})}^2+\|V\|_{W^{\frac14,2}(0,T;\Hb^{\s-1})}^2\leq R^2 \right\},\]
    and $B_R^2=B_R^{2,1}+B_R^{2,2}$, where $B_R^{2,1}$ and $B_R^{2,2}$ are the closed balls in $W^{1,2}(0,T;\Hb^{\s-2+\frac{s}{2}})$ and $W^{\alpha,p}(0,T;\Hb^{\s-2+\frac{s}{2}})$ centered at $0$ with radius $R$, then $B_R^1\cap B_R^2$ is precompact in $L^2(0,T;\Hb^{\s})\cap C([0,T];\Hb^{\s-2})$. By Markov's inequality and Proposition~\ref{t.062502}, one has
    \begin{align*}
        \mu_V^n\left((B_R^1)^c\right)&\leq \Pb\left(\|V_n\|_{L^2(0,T;\Hb^{\s+\frac{s}{2}})}\geq \frac{R}{\sqrt{2}}\right)+\Pb\left(\|V_n\|_{W^{\frac14,2}(0,T;\Hb^{\s-1})}\geq \frac{R}{\sqrt{2}}\right)\\
        &\leq \frac{C(1+\Eb\|V_0\|_{\s}^2)}{R},
    \end{align*}
    and 
    \begin{align*}
        \mu_V^n\left((B_R^2)^c\right)&\leq \Pb\left(\left\|V_n-\int_0^{\cdot}\sum_kP_n\Pc B_k V_ndW_r^k\right\|_{W^{1,2}(0,T;\Hb^{\s-2+\frac{s}{2}})}\geq R\right)\\
        &\qquad +\Pb\left(\left\|\int_0^{\cdot}\sum_kP_n\Pc B_k V_ndW_r^k\right\|_{W^{\alpha,p}(0,T;\Hb^{\s-2+\frac{s}{2}})}^p\geq R^p\right)\\
        &\leq \frac{C(1+\Eb\|V_0\|_{\s}^p)}{R}. 
    \end{align*}
    Thus the sequence  $\lbrace \mu^n \rbrace_{n\geq 1}$ is tight over $\mathcal{X}$. By Prokhorov's theorem, the sequence is precompact. Applying the Skorokhod theorem yields the existence of a probability space $(\widetilde{\Omega}, \widetilde{\mathcal{F}}, \widetilde{\mathbb{P}})$, a subsequence $n_j \rightarrow \infty$ as $j \rightarrow \infty$, and a sequence of $\mathcal{X}$-valued random variables $\left(\widetilde{V}^0_{n_j},\widetilde{V}_{n_j}, \widetilde{\mathbb W}_{n_j}\right)$ such that:
    \begin{itemize}
        \item The sequence converges almost surely under $\widetilde{\mathbb P}$ to $\left(\widetilde{V}_0,\widetilde{V}, \widetilde{\mathbb W}\right)$ in $\mathcal{X}$;
        \item Each triple $\left(\widetilde{V}^0_{n_j},\widetilde{V}_{n_j}, \widetilde{\mathbb W}_{n_j}\right)$ is a martingale solution to \eqref{e.052505} for $n=n_j$ with initial data $\widetilde{V}^0_{n_j}$.
        \item The law of $\widetilde{V}^0$ coincides with that of $V_0$.
    \end{itemize}
    Moreover, the sequence $\widetilde{V}_{n_j}$ satisfies the same energy estimate under the new probability space as in Proposition~\ref{t.062502}. 
    
    \medskip

\noindent\textbf{Step 2: Identify the limit as a martingale solution.} We now show that the limit $\left(\widetilde{V}_0,\widetilde{V}, \widetilde{\mathbb W}\right)$ is the desired martingale solution over the stochastic basis $(\widetilde{\Omega},\widetilde{\mathcal{F}}, \widetilde{\mathbb{F}}, \widetilde{\mathbb{P}})$ with $\widetilde{\mathbb{F}}$ being the filtration generated by $\widetilde{V}, \widetilde{\Wb}$. Firstly, by the convergence $\widetilde{V}_{n_j}\to \widetilde{V}$ in $\mathcal{X}$, Proposition~\ref{t.062502} and the Banach-Alaoglu theorem, we infer:
    \begin{align}\label{e.071501}
        \widetilde{V}\in L^2\left(\widetilde{\Omega} ; L^2\left(0, T ; \Hb^{\s+\frac{s}{2}}\right) \cap L^{\infty}\left(0, T ; \Hb^{\s}\right)\right). 
    \end{align}
    Moreover, the Vitali convergence theorem implies:
    \begin{align}\label{eq.L060102}
        \widetilde{V}_{n_j} \rightarrow \widetilde{V} \text { in } L^2\left(\widetilde{\Omega}; L^2\left(0, T; \Hb^{\s}\right)\right).  
    \end{align}
    Thus there exists a subsequence, still denote by $\widetilde{V}_{n_j}$, such that: 
    \begin{align}\label{eq.L060103}
        \widetilde{V}_{n_j} \rightarrow \widetilde{V} \text { in } L^2\left(0, T; \Hb^{\s}\right) \text { for a.a. }\omega \in\widetilde{\Omega}\text {. }
    \end{align}
    Since each pair $(\widetilde{V}_{n_j}, \widetilde{\mathbb W}_{n_j})$ is a martingale solution to system \eqref{e.052505} with $n=n_j$, we have  
    \begin{align}\label{eq.L060101}
        \begin{split}
            &\left\langle \widetilde{V}_{n_j}(t),\phi \right\rangle+ \int_0^t\left\langle\theta_\rho^2P_{n_j}\Pc Q(\widetilde{V}_{n_j}, \widetilde{V}_{n_j}) + P_{n_j}\Pc F(\widetilde{V}_{n_j})+\LL^s\widetilde{V}_{n_j}, \phi\right\rangle dr\\
            &\qquad = \left\langle \widetilde{V}_{n_j}(0),\phi \right\rangle+ \frac12 \sum_k\int_0^t\left\langle P_{n_j}(\Pc B_k)^2\widetilde{V}_{n_j},\phi \right\rangle dr  \\ 
            &\qquad\qquad + \sum_k \int_0^t\left\langle P_{n_j}\Pc B_k\widetilde{V}_{n_j}, \phi \right\rangle d\widetilde{W}_r^{k,n_j} , 
        \end{split}
    \end{align}
    for any $\phi\in \Hb$ and $t\in [0,T]$. Here and below, $\widetilde{W}^{k,n_j}$ is the $k$-th component of the Wiener process $\widetilde{\mathbb W}_{n_j}$.  We now prove that $(\widetilde{V},\widetilde{\mathbb W})$ is a martingale solution to \eqref{e.052504} by passing the limit $j\to\infty$ in \eqref{eq.L060101}. 

\smallskip
{\it Linear Terms.}    The convergence of the linear terms follows straightforwardly from \eqref{eq.L060103}. Here, we only address the linear term corresponding to the It\^o-Stratonovich corrector. Observe that 
    \begin{align*}
        &\left|\frac12 \sum_k\int_0^t\left\langle P_{n_j}(\Pc B_k)^2\widetilde{V}_{n_j},\phi \right\rangle dr - \frac12 \sum_k\int_0^t\left\langle (\Pc B_k)^2\widetilde{V},\phi \right\rangle dr\right|\\
        &\leq  \frac12 \sum_k\int_0^t\left|\left\langle (\Pc B_k)^2(\widetilde{V}_{n_j}-\widetilde{V}),P_{n_j}\phi \right\rangle \right|dr + \frac12 \sum_k\int_0^t\left|\left\langle (\Pc B_k)^2\widetilde{V},P_{n_j}\phi -\phi\right\rangle \right|dr\\
        &\leq C\|b\|_{\ell^2(\Nb,W^{1,\infty})}^2\left(\|\phi\|\int_0^T\|\widetilde{V}_{n_j}-\widetilde{V}\|_{2}dr + \|P_{n_j}\phi-\phi\|\int_0^T\|\widetilde{V}\|_2dr\right)\to 0,
    \end{align*}
    almost surely by \eqref{e.071501} and \eqref{eq.L060103}. 

\smallskip
{\it Nonlinear Terms. }    Denote $\theta_\rho^2(V) = \theta_\rho^2(\|V\|_{W^{1,\infty}})$ for convenience. For the nonlinear term, we have 
    \begin{align*}
        \begin{split}
            &\left|\int_0^t\left\langle\theta_\rho^2(\widetilde{V}_{n_j})P_{n_j}\Pc Q(\widetilde{V}_{n_j}, \widetilde{V}_{n_j}),\phi\right\rangle dr-\int_0^t\left\langle\theta_\rho^2(\widetilde{V})\Pc Q(\widetilde{V}, \widetilde{V}),\phi\right\rangle dr\right|\\
            &\leq \int_0^t\left|\left\langle\theta_\rho^2(\widetilde{V}_{n_j})\Pc Q(\widetilde{V}_{n_j}, \widetilde{V}_{n_j}),P_{n_j}\phi-\phi\right\rangle \right|dr+\int_0^t\left|\left\langle\left(\theta_\rho^2(\widetilde{V}_{n_j})-\theta_\rho^2(\widetilde{V})\right)\Pc Q(\widetilde{V}_{n_j}, \widetilde{V}_{n_j}),\phi\right\rangle \right|dr\\
            &\qquad +\int_0^t\left|\left\langle\theta_\rho^2(\widetilde{V})\left(\Pc Q(\widetilde{V}_{n_j}, \widetilde{V}_{n_j})-\Pc Q(\widetilde{V}, \widetilde{V}) \right),\phi\right\rangle \right|dr = N_1+N_2+N_3. 
        \end{split}
    \end{align*}
    By the Cauchy-Schwarz inequality, Lemma \ref{l.071701} and Sobolev embedding we have
    \begin{align*}
        N_1\leq C\|P_{n_j}\phi-\phi\|\int_0^t\|Q(\widetilde{V}_{n_j}, \widetilde{V}_{n_j})\|dr\leq C\|P_{n_j}\phi-\phi\|\int_0^t\|\widetilde{V}_{n_j}\|_{\s}^2dr\to 0,
    \end{align*}
    almost surely. Similarly, one has 
    \begin{align*}
        N_3&\leq C\|\phi\|\int_0^t\|\widetilde{V}_{n_j}-\widetilde{V}\|_{\s}(\|\widetilde{V}_{n_j}\|_{\s}+\|\widetilde{V}\|_{\s})dr\\
        &\leq C\|\phi\|\left(\int_0^T\|\widetilde{V}_{n_j}-\widetilde{V}\|_{\s}^2dr\right)^{1/2}\left(\int_0^T\|\widetilde{V}_{n_j}\|_{\s}^2+\|\widetilde{V}\|_{\s}^2dr\right)^{1/2}\to 0,
    \end{align*}
    almost surely. By H\"older's inequality, the Lipschitz continuity of the cutoff function, Lemma \ref{l.071701}, and Sobolev embeddings, we have 
    \begin{align*}
        N_2&\leq C\left(\int_0^T\|\widetilde{V}_{n_j}-\widetilde{V}\|_{\s}^2dr\right)^{\frac12}\left(\int_0^T\|Q(\widetilde{V}_{n_j},\widetilde{V}_{n_j})\|^2\|\phi\|^2dr\right)^{\frac12}\\
        &\leq C\|\phi\|\left(\int_0^T\|\widetilde{V}_{n_j}-\widetilde{V}\|_{\s}^2dr\right)^{\frac12}\left(\int_0^T\|\widetilde{V}_{n_j}\|_{W^{1,\infty}}^2\|\widetilde{V}_{n_j}\|_{1}^2dr\right)^{\frac12}\\
        &\leq C\|\phi\|\left(\int_0^T\|\widetilde{V}_{n_j}-\widetilde{V}\|_{\s}^2dr\right)^{\frac12}\left(\int_0^T\|\widetilde{V}_{n_j}\|_{\s}^2dr\right)^{\frac12}\sup_{r\in [0,T]}\|\widetilde{V}_{n_j}\|_{\s-2}\to 0,
    \end{align*}
    almost surely by \eqref{eq.L060102} and the convergence of $\widetilde{V}_{n_j}\to \widetilde{V}$ in $\mathcal{X}$ (ensuring that $\sup_{r\in [0,T]}\|\widetilde{V}_{n_j}\|_{\s-2}$ remains finite almost surely in the limit $j\to\infty$). 

\smallskip
{\it Stochastic Terms.} Lastly, we look at the stochastic integral term. Since $\s>3$ and $\widetilde{V}_{n_j}\to \widetilde{V}$ in $C([0,T];\Hb^{\s-2})$, we have 
    \begin{align*}
        &\sup_{r\in[0,T]}\left|\sum_k\left\langle P_{n_j}\Pc B_k\widetilde{V}_{n_j}, \phi \right\rangle - \sum_k\left\langle \Pc B_k\widetilde{V}, \phi \right\rangle\right|\\
        &\leq  C\sup_{r\in[0,T]}\sum_k \left(\left|\left\langle \Pc B_k\widetilde{V}_{n_j}, P_{n_j}\phi -\phi\right\rangle\right|+\left|\left\langle \Pc B_k\left(\widetilde{V}_{n_j}-\widetilde{V}\right), \phi\right\rangle\right|\right)\\
        &\leq C\|b\|_{\ell^2(\Nb,L^{\infty})}\left(\|P_{n_j}\phi -\phi\|\sup_{r\in[0,T]}\|\widetilde{V}_{n_j}(r)\|_{\s-2}+\|\phi\|\sup_{r\in[0,T]}\|\widetilde{V}_{n_j}-\widetilde{V}\|_{\s-2}\right)\to 0,
    \end{align*} 
    almost surely. This fact, together with the almost sure convergence of 
    $\widetilde{\mathbb W}_{n_j} \to \widetilde{\mathbb W}$ in  $ C\left(\left[0, T\right]; \Uc \right)$ and Theorem 4.2 in \cite{kurtz1996weak} imply the following convergence 
    \[\sum_k \int_0^t\left\langle P_{n_j}\Pc B_k\widetilde{V}_{n_j}, \phi \right\rangle d\widetilde{W}_r^{k,n_j} \to \sum_k \int_0^t\left\langle \Pc B_k\widetilde{V}, \phi \right\rangle d\widetilde{W}_r^{k}\]
    in probability in $C([0,T];\mathbb R)$. Passing to a subsequence if necessary, we obtain the above convergence in $C([0,T];\mathbb R)$ almost surely. 
    
    We have just proved that for each $\phi\in \Hb$ and $t\in[0,T]$, the triple $(\widetilde{V}_0,\widetilde{V},\widetilde{\Wb})$ satisfies 
    \begin{align*}
        \begin{split}
            \left\langle \widetilde{V}(t),\phi \right\rangle+ \int_0^t\left\langle\theta_\rho^2\Pc Q(\widetilde{V}, \widetilde{V}) + \Pc F(\widetilde{V})+\LL^s\widetilde{V}, \phi\right\rangle dr 
            &= \left\langle \widetilde{V}(0),\phi \right\rangle+ \frac12 \sum_k\int_0^t\left\langle (\Pc B_k)^2\widetilde{V},\phi \right\rangle dr  \\ 
            &\qquad\qquad + \sum_k \int_0^t\left\langle \Pc B_k\widetilde{V}, \phi \right\rangle d\widetilde{W}_r^{k} , 
        \end{split}
    \end{align*}
    almost surely over $(\widetilde{\Omega},\widetilde{\mathcal{F}}, \widetilde{\mathbb{P}})$. This result shows that $t\to\left\langle \widetilde{V}(t),\phi \right\rangle$ is continuous for each  $\phi\in \Hb$.  It remains to prove that $\widetilde{V}\in L^2(\widetilde{\Omega};C([0,T],\Hb^{\s}))$ so that $\left(\widetilde{V}_0,\widetilde{V}, \widetilde{\mathbb W}\right)$ is the desired martingale solution.

    First, we note that each solution sample path $t\to \widetilde{V}(t)$ is weakly continuous in $\Hb^{\s}$. Indeed, since $\widetilde V\in L^2(\widetilde{\Omega};L^{\infty}([0,T],\Hb^{\s}))$, its $L^{\infty}([0,T],\Hb^{\s})$ norm is almost surely finite. For any $\phi\in \Hb^{-\s}$, by the density of $C^{\infty}(\Tb^3)\cap \Hb$ in $\Hb^{-\s}$,  there exists a sequence $\phi_n\in C^{\infty}(\Tb^3)\cap\Hb$ such that $\phi_n \to \phi$ in $\Hb^{-\s}$. Therefore,
    \[\sup_{t\in[0,T]}\left|\left\langle \widetilde{V}(t),\phi_n \right\rangle-\left\langle \widetilde{V}(t),\phi \right\rangle\right|\leq \|\widetilde{V}\|_{L^{\infty}([0,T],\Hb^{\s})}\|\phi_n-\phi\|_{-\s}\to 0. \]
    Consequently, $t\to\left\langle \widetilde{V}(t),\phi \right\rangle$ is continuous for each  $\phi\in \Hb^{-\s}$. Therefore, to prove $\widetilde{V}\in L^2(\widetilde{\Omega};C([0,T],\Hb^{\s}))$, we only need to show that $t\to \|\widetilde{V}(t)\|_{\s}$ is continuous almost surely. Direct application of It\^o's formula is not possible for the norm $\|\cdot\|_{\s}$, due to the limited regularity of the solution and the It\^o-Stratonovich corrector. We instead use a mollification argument, adapted from \cite{kim2009existence}. 
    
    Let $\varepsilon>0$, and define the mollification operator $J_{\varepsilon}$ using the standard mollifier $\phi_{\varepsilon}$ on the periodic torus with respect to the spatial variable: 
    \[J_{\varepsilon}f(x) = (\phi_{\varepsilon}\ast f)(x).\]
    The operator is self-adjoint on $H$ and commutes with $\Pc$ and $\LL$ as a Fourier multiplier. Applying $J_{\varepsilon}$ to the equation satisfied by $\widetilde{V}$, we obtain   
    \begin{align*}
            \begin{split}
                J_{\varepsilon}\widetilde{V}(t)+ \int_0^tJ_{\varepsilon}\left(\theta_\rho^2\Pc Q(\widetilde{V}, \widetilde{V}) + \Pc F(\widetilde{V})+\LL^s\widetilde{V}\right) dr 
                &= J_{\varepsilon}\widetilde{V}_0+ \frac12 \sum_k\int_0^tJ_{\varepsilon} (\Pc B_k)^2\widetilde{V} dr  \\ 
                &\qquad\qquad + \sum_k \int_0^tJ_{\varepsilon}\Pc B_k\widetilde{V} d\widetilde{W}_r^{k}.  
            \end{split}
    \end{align*}
    Applying It\^o's formula to $\|\LL^{\s}J_{\varepsilon}\widetilde{V}(t)\|^{2}$ yields 
    \begin{align*}
        \begin{split}
            \|\LL^{\s}J_{\varepsilon}\widetilde{V}(t)\|^{2} &=\|\LL^{\s}J_{\varepsilon}\widetilde{V}_0\|^{2} -\int_0^t2\left\langle J_{\varepsilon}\left(\theta_\rho^2 \Pc Q(\widetilde{V},\widetilde{V})+\Pc F(\widetilde{V})+\Lambda^s \widetilde{V}\right),\Lambda^{2\sigma}J_{\varepsilon}\widetilde{V}\right\rangle dr \\
            &+\sum_{k}\int_0^t\left(\left\langle J_{\varepsilon}(\mathcal{P}B_k)^2\widetilde{V},\Lambda^{2\sigma}J_{\varepsilon}\widetilde{V} \right\rangle+\|\LL^{\s}J_{\varepsilon}\Pc B_k\widetilde{V}\|^2\right) dr+\int_0^t2\sum_{k}\langle J_{\varepsilon} \Pc B_k \widetilde{V}, \LL^{2\sigma}J_{\varepsilon} \widetilde{V}\rangle d\widetilde{W}_r^{k}.
        \end{split}
    \end{align*}
    Estimate similarly as in the proof of Proposition~\ref{t.062502}, we have 
    \begin{align*}
        \left|\left\langle J_{\varepsilon}\left(\theta_\rho^2 \Pc Q(\widetilde{V},\widetilde{V})+\Pc F(\widetilde{V})+\Lambda^s \widetilde{V}\right),\Lambda^{2\sigma}J_{\varepsilon}\widetilde{V}\right\rangle\right|\leq (C_{p,\s,\rho}+1)\|\widetilde{V}\|_{\s+\frac{s}{2}}^2. 
    \end{align*}
    Denote by $A^{\s}=\Lambda^{\s}J_{\varepsilon}$ and perform estimates as in \eqref{e.070701} gives
    \begin{align*}
        &\left|\left\langle J_{\varepsilon}(\mathcal{P}B_k)^2\widetilde{V},\Lambda^{2\sigma}J_{\varepsilon}\widetilde{V} \right\rangle+\|\LL^{\s}J_{\varepsilon}\Pc B_k\widetilde{V}\|\right|\\
        &=\Big|\langle [A^{\s},B_k ][\Pc,  B_k] \widetilde{V},A^{\s}\widetilde{V} \rangle + \langle A^{\s}[\Pc,  B_k] \widetilde{V},[A^{\s},B_k ] \widetilde{V} \rangle\\
        &\qquad+ \langle  [[A^{\s},B_k ],B_k]\widetilde{V},A^{\s}\widetilde{V} \rangle + \langle [A^{\s},B_k ] \widetilde{V}, [A^{\s},B_k ] \widetilde{V}\rangle\Big|.
    \end{align*}
    Similar to \eqref{e.072103}, under the assumption on the noise coefficients we have 
    \begin{align*}
        \left|\sum_{k}\int_0^t\left(\left\langle J_{\varepsilon}(\mathcal{P}B_k)^2\widetilde{V},\Lambda^{2\sigma}J_{\varepsilon}\widetilde{V} \right\rangle+\|\LL^{\s}J_{\varepsilon}\Pc B_k\widetilde{V}\|^2\right) dr\right|\leq \int_0^t\|\widetilde{V} \|_{\s+\frac{s}{2}}^2dr+C\int_0^t\|\widetilde{V} \|_{\s}^2dr.
    \end{align*}

    On the other hand, since $J_{\varepsilon}$ is self-adjoint and each $b_k$ in $B_k=b_k\cdot\nabla$ is divergence-free, the Burkholder-Davis-Gundy inequality, Minkowski's inequality, Lemma~\ref{l.071701}, and Sobolev's inequality yield:
    \begin{align*}
        &\Eb\sup_{t\in[0,T]}\left|\int_0^t2\sum_{k}\left(\left\langle J_{\varepsilon} \Pc B_k \widetilde{V}, \LL^{2\sigma}J_{\varepsilon} \widetilde{V}\right\rangle-\left\langle  \Pc B_k \widetilde{V}, \LL^{2\sigma} \widetilde{V}\right\rangle \right)d\widetilde{W}_r^{k}\right|\\
        &=2\Eb\sup_{t\in[0,T]}\left|\int_0^t\sum_{k}\left\langle \LL^{\s} B_k \widetilde{V}, \LL^{\sigma}J_{\varepsilon}^2 \widetilde{V}-\LL^{\sigma}\widetilde{V}\right\rangle d\widetilde{W}_r^{k}
        \right|\\
        &\leq C \Eb\left(\int_0^T \sum_{k}\left\langle \LL^{\s} B_k \widetilde{V}, \LL^{\sigma}J_{\varepsilon}^2 \widetilde{V}-\LL^{\sigma}\widetilde{V}\right\rangle^2  dr\right)^{\frac12}\\
        &\leq C \Eb\left(\int_0^T \sum_{k}\left\langle [\LL^{\s}, B_k ]\widetilde{V}, \LL^{\sigma}J_{\varepsilon}^2 \widetilde{V}-\LL^{\sigma}\widetilde{V}\right\rangle^2  dr\right)^{\frac12}+C \Eb\left(\int_0^T \sum_{k}\left\langle [J_{\varepsilon}, B_k]\LL^{\s}\widetilde{V}, \LL^{\sigma}J_{\varepsilon} \widetilde{V}\right\rangle^2  dr\right)^{\frac12}\\
        &\lesssim \delta\Eb\sup_{t\in[0,T]}\|\widetilde{V}\|_{\s}^2+C_{\delta}\Eb\left(\int_0^T\|\LL^{\sigma}J_{\varepsilon}^2 \widetilde{V}-\LL^{\sigma}\widetilde{V}\|^2dr + \int_0^T\|[J_{\varepsilon}, B_k]\LL^{\s}\widetilde{V}\|^2 dr\right)\to 0,
    \end{align*}
    by first letting $\varepsilon\to 0$ then $\delta\to0$. Here, the convergence of the last integral follows from the Friedrichs' lemma \cite{hormander2007analysis,taylor2017pseudodifferential}. 
    Thus, there exists a sequence $\varepsilon_n\to 0$ such that as $n\to \infty$, 
    \[\int_0^{\cdot}2\sum_{k}\langle J_{\varepsilon_n} \Pc B_k \widetilde{V}, \LL^{2\sigma}J_{\varepsilon_n} \widetilde{V}\rangle d\widetilde{W}_r^{k}\to \int_0^{\cdot}2\sum_{k}\langle  \Pc B_k \widetilde{V}, \LL^{2\sigma}\widetilde{V}\rangle d\widetilde{W}_r^{k},\]
    almost surely in $C([0,T];\Rb)$. 
    
    Consequently, for any $0\leq t_1<t_2\leq T$, we have almost surely 
    \begin{align}\label{e.071502}
        \begin{split}
        \Big|\|\widetilde{V}(t_2)\|_{\sigma}^{2}&-\|\widetilde{V}(t_1)\|_{\sigma}^{2}\Big| = \lim_{n\to\infty}\left|\|J_{\varepsilon_n}\widetilde{V}(t_2)\|_{\sigma}^{2}-\|J_{\varepsilon_n}\widetilde{V}(t_1)\|_{\sigma}^{2}\right|\\
        &\leq \int_{t_1}^{t_2}\left(2(C_{p,\s,\rho}+2)\|\widetilde{V}\|_{\s+\frac{s}{2}}^2+C\|\widetilde{V}\|_{\s}^2\right)dr + \left|\int_{t_1}^{t_2}2\sum_{k}\langle  \Pc B_k \widetilde{V}, \LL^{2\sigma}\widetilde{V}\rangle d\widetilde{W}_r^{k}\right|. 
        \end{split}
    \end{align}
    From \eqref{e.071501}, we know that $\widetilde{V}\in L^2\left(0, T ; \Hb^{\s+\frac{s}{2}}\right) \cap L^{\infty}\left(0, T ; \Hb^{\s}\right)$ almost surely. Therefore \eqref{e.071502} implies the continuity of $t\to \|\widetilde{V}(t)\|_{\s}$.  Combining this with the weak continuity proved earlier, we conclude that $\widetilde{V}\in C([0,T],\Hb^{\s})$ almost surely, and hence $\widetilde{V}\in L^2(\Omega; C([0,T],\Hb^{\s}))$.
\end{proof}

\subsection{Pathwise uniqueness and local pathwise solutions}\label{s.121101}
In this subsection, we establish the pathwise uniqueness of martingale solutions to the cutoff system \eqref{e.052504}. Combining this result with the existence of a martingale solution (Proposition~\ref{p.070601}), we then deduce the existence of a unique pathwise solution.

\begin{proposition}\label{p.121001}
    Assume the same conditions as in Proposition~\ref{t.062502}. 
   Let $(\Sc, \Wb, V_1)$ and $(\Sc, \Wb, V_2)$ be two solutions to the cutoff system \eqref{e.052504} on $[0,T]$ with the same initial data $V_0$, over the same stochastic basis $\Sc$ and driven by the same noise $\Wb$. Then 
    \[\Pb\left(V_1=V_2, \text{ for all } t\in [0,T]\right) =1.\] 
\end{proposition}
\begin{proof}
    Let $\bV = V_1-V_2$. Denote $\|\cdot\|_{1,\infty}=\|\cdot\|_{W^{1,\infty}}$ for convenience. Note that $\bV$ solves 
    \begin{align*}
        \begin{split}
            d\bV + \big[\theta_\rho^2(\|V_1\|_{1,\infty})\mathcal{P}Q( V_1, V_1)-\theta_\rho^2(\|V_2\|_{1,\infty})\mathcal{P}Q( V_2, V_2) &+ \mathcal{P}F(\bV) + \Lambda^s \bV\big]dt\\
            &= \frac12\sum_{k=1}^{\infty}(\mathcal{P}B_k)^2\bV +  \sum_{k=1}^\infty\mathcal{P}B_k\bV dW^k,
        \end{split}
    \end{align*}  
    with initial data $\bV_0=0$. By It\^o's formula, we have 
    \begin{align*}
        \frac12d\|\bV\|^2 + \Big\langle\big[\theta_\rho^2(\|V_1\|_{1,\infty})\mathcal{P}Q( V_1, V_1)&-\theta_\rho^2(\|V_2\|_{1,\infty})\mathcal{P}Q( V_2, V_2)  + \Lambda^s \bV\big], \bV\Big\rangle dt\\
        & = \frac12\sum_{k=1}^{\infty}\left(\left\langle(\mathcal{P}B_k)^2\bV, \bV\right\rangle + \|\Pc B_k \bV\|^2\right)dt + \sum_{k=1}^\infty\langle\mathcal{P}B_k\bV, \bV\rangle dW^k.
    \end{align*}
   Using integration by parts, one has
    $\langle \Pc Q(V_2, \bV), \bV\rangle=\langle\mathcal{P}B_k\bV, \bV\rangle=0$, and $\langle(\mathcal{P}B_k)^2\bV, \bV\rangle = - \|\mathcal{P} B_k\bV\|^2$. Thus, the above equation reduces to 
    \begin{align*}
        \frac12d\|\bV\|^2+ \Big\langle\left(\theta_\rho^2(\|V_1\|_{1,\infty})-\theta_\rho^2(\|V_2\|_{1,\infty})\right)\mathcal{P}Q( V_1, V_1)
        &+\theta_\rho^2(\|V_2\|_{1,\infty})\mathcal{P}Q( \bV, V_1), \bV\Big\rangle dt\\
        & = -\|\LL^{\frac{s}{2}}\bV\|^2dt.
    \end{align*}
    For $\s>3$, using Lemma \ref{l.071701} and Sobolev's inequality $\|f\|_{1,\infty}\lesssim \|f\|_{\s-\frac12}$, we estimate 
    \begin{align*}
        \Big\langle\left(\theta_\rho^2(\|V_1\|_{1,\infty})-\theta_\rho^2(\|V_2\|_{1,\infty})\right)\mathcal{P}Q( V_1, V_1), \bV\Big\rangle 
        \leq C\|\bV\|_{1,\infty}\|\bV\|\|V_1\|_{1,\infty}\|V_1\|_{1}\leq C\|\bV\|_{\s-\frac12}^2\|V_1\|_{\s-\frac12}^2, 
    \end{align*}
    and  
    \begin{align*}
        \Big\langle\theta_\rho^2(\|V_2\|_{1,\infty})\mathcal{P}Q( \bV, V_1), \bV\Big\rangle\leq C\left(\|\bV\|_{1,\infty}\|V_1\|_{1}+\|\bV\|_{1}\|\|V_1\|_{1,\infty}\right)\|\bV\|\leq C\|\bV\|_{\s-\frac12}^2\|V_1\|_{\s-\frac12}. 
    \end{align*}
    Consequently, we obtain 
    \begin{align}\label{e.072301}
        \frac12d\|\bV\|^2+\|\LL^{\frac{s}{2}}\bV\|^2dt\leq C\|\bV\|_{\s-\frac12}^2\left(\|V_1\|_{\s-\frac12}^2+1\right) dt.
    \end{align}

    Next, we estimate the $\s-\frac12$ norm of $\bV$. By It\^o's formula we have 
    \begin{align*}
        \frac12d\|\LL^{\s-\frac12}\bV\|^2 &+ \Big\langle \LL^{\s-\frac12}\big[\theta_\rho^2(\|V_1\|_{1,\infty})\mathcal{P}Q( V_1, V_1)-\theta_\rho^2(\|V_2\|_{1,\infty})\mathcal{P}Q( V_2, V_2)  + \Lambda^s \bV\big], \LL^{\s-\frac12}\bV\Big\rangle dt\\
        & = \frac12\sum_{k=1}^{\infty}\left(\left\langle \LL^{\s-\frac12}(\mathcal{P}B_k)^2\bV, \LL^{\s-\frac12}\bV\right\rangle + \|\LL^{\s-\frac12}\Pc B_k \bV\|^2\right)dt + \sum_{k=1}^\infty\langle \LL^{\s-\frac12}\mathcal{P}B_k\bV, \LL^{\s-\frac12}\bV\rangle dW^k.
    \end{align*} 
    Following arguments similar to to \eqref{s>1} and \eqref{s=1}, we have 
    \begin{align*}
        \frac12\sum_{k=1}^{\infty}\left(\left\langle \LL^{\s-\frac12}(\mathcal{P}B_k)^2\bV, \LL^{\s-\frac12}\bV\right\rangle + \|\LL^{\s-\frac12}\Pc B_k \bV\|^2\right)
        \leq \frac14\|\LL^{\s-\frac12+\frac{s}{2}}\bV\|^2+C\|\bV\|_{\s-\frac12}^2. 
    \end{align*}
    For the nonlinear term, by utilizing the double cutoff design, we first rewrite
    \begin{align*}
        &\theta_\rho^2(\|V_1\|_{1,\infty})\mathcal{P}Q( V_1, V_1)-\theta_\rho^2(\|V_2\|_{1,\infty})\mathcal{P}Q( V_2, V_2)\\
        &= \Big(\theta_\rho(\|V_1\|_{1,\infty})-\theta_\rho(\|V_2\|_{1,\infty})\Big)\Big(\theta_\rho(\|V_1\|_{1,\infty})\mathcal{P}Q( V_1, V_1)+\theta_\rho(\|V_2\|_{1,\infty})\mathcal{P}Q( V_2, V_2)\Big)\\
        &\qquad +\theta_\rho(\|V_2\|_{1,\infty})\theta_\rho(\|V_1\|_{1,\infty})\mathcal{P}\left(Q( V_1, V_1)-Q( V_2, V_2)\right): = Q_1+Q_2,
    \end{align*}
    where the identity 
    $a_1^2b_1-a_2^2b_2 = (a_1-a_2)(a_1b_1+a_2b_2)+a_1a_2(b_1-b_2)$ has been used. 
    By the Lipschitz continuity of the cutoff function, Lemma \ref{l.071701} and Sobolev's inequality, we have 
    \begin{align*}
        \Big\langle \LL^{\s-\frac12}Q_1, \LL^{\s-\frac12}\bV\Big\rangle\leq \|\LL^{\s-\frac12}Q_1\|\|\LL^{\s-\frac12}\bV\|\leq \|\bV\|_{\s-\frac12}^2\left(\|V_1\|_{\s+\frac12}^2+\|V_2\|_{\s+\frac12}^2\right).
    \end{align*}
    Thanks to the property of the cutoff function, we deduce 
    \begin{align*}
        \Big\langle \LL^{\s-\frac12}Q_2, \LL^{\s-\frac12}\bV\Big\rangle
        &=\theta_\rho(\|V_2\|_{1,\infty})\theta_\rho(\|V_1\|_{1,\infty})\left\langle\LL^{\s-1}\left[\left(Q( V_1, \bV)+Q( \bV, V_2)\right)\right], \LL^{\s}\bV\right\rangle\\
        &\leq C_{\s}\left(\rho\|\LL^{\s}\bV\|+\left(\|\LL^{\s}V_1\|+\|\LL^{\s}V_2\|\right)\|\bV\|_{1,\infty}\right)\|\LL^{\s}\bV\|\\
        &\leq C_{\s}\rho\|\LL^{\s}\bV\|^2+ C_{\rho}\left(\|V_1\|_{\s}^2+\|V_2\|_{\s}^2\right)\|\bV\|_{\s-\frac12}^2, 
    \end{align*}
    where we used Young's inequality for products at the final step. 
    
    Now we need to distinguish between $s\in(1,2]$ and $s=1$. For $s>1$ we infer from the interpolation inequality $\|\LL^{\s}\bV\|^2\leq \varepsilon\|\LL^{\s-\frac12+\frac{s}{2}}\bV\|^2+C_{\varepsilon}\|\LL^{\s-\frac12}\bV\|^2$ that 
    \begin{align}\label{e.071702}
        \Big\langle \LL^{\s-\frac12}Q_2, \LL^{\s-\frac12}\bV\Big\rangle\leq \frac12\|\LL^{\s-\frac12+\frac{s}{2}}\bV\|^2+C\left(1+\|V_1\|_{\s}^2+\|V_2\|_{\s}^2\right)\|\bV\|_{\s-\frac12}^2,
    \end{align}
    by choosing appropriate $\varepsilon$. For $s=1$, we impose the  smallness condition $\rho C_{\s}<\frac12$, ensuring the same bound \eqref{e.071702} holds. 
    Combining the above estimates, we obtain 
    \begin{align*}
        \frac12d\|\LL^{\s-\frac12}\bV\|^2 &+\frac14\|\LL^{\s-\frac12+\frac{s}{2}}\bV\|^2 dt\\
        &\leq C\left(1+\|V_1\|_{\s+\frac12}^2+\|V_2\|_{\s+\frac12}^2\right)\|\bV\|_{\s-\frac12}^2dt + \sum_{k=1}^\infty\langle \LL^{\s-\frac12}\mathcal{P}B_k\bV, \LL^{\s-\frac12}\bV\rangle dW_t^k. 
    \end{align*} 
    This together with estimate \eqref{e.072301} yields
    \begin{align*}
        d\|\bV\|_{\s-\frac12}^2\leq C\left(1+\|V_1\|_{\s+\frac12}^2+\|V_2\|_{\s+\frac12}^2\right)\|\bV\|_{\s-\frac12}^2dt + 2\sum_{k=1}^\infty\langle \LL^{\s-\frac12}\mathcal{P}B_k\bV, \LL^{\s-\frac12}\bV\rangle dW_t^k.  
    \end{align*}
    Letting
    \[Y_t = \exp\left(-C\int_0^t\left(1+\|V_1\|_{\s+\frac12}^2+\|V_2\|_{\s+\frac12}^2\right)dr\right),\]
    It\^o's formula gives
    \begin{align*}
        d\left(Y_t\|\bV\|_{\s-\frac12}^2\right)\leq 2 Y_t\sum_{k=1}^\infty\langle \LL^{\s-\frac12}\mathcal{P}B_k\bV, \LL^{\s-\frac12}\bV\rangle dW_t^k.
    \end{align*}
    In the integral form, this reads 
    \begin{align*}
        Y_t\|\bV(t)\|_{\s-\frac12}^2 \leq  \|\bV_0\|_{\s-\frac12}^2+2\sum_{k=1}^\infty\int_0^tY_r\langle \LL^{\s-\frac12}\mathcal{P}B_k\bV, \LL^{\s-\frac12}\bV\rangle dW_r^k. 
    \end{align*}
    Since $V_0=0$ and the stochastic integral is a martingale, we obtain 
    \begin{align*}
        \Eb\left[Y_t\|\bV(t)\|_{\s-\frac12}^2 \right]\leq 0. 
    \end{align*}
    As $0<Y_t\leq 1$, it follows that $\|\bV(t)\|_{\s-\frac12}^2 =0$ almost surely for all $t$. Because $V_1$ and $V_2$ are modifications both with continuous sample paths, they are indistinguishable. This completes the proof of pathwise uniqueness. 
\end{proof}

We are now ready to prove the main result, Theorem \ref{t.062501}, concerning the existence of a unique maximal pathwise solution. 

\begin{proof}[Proof of Theorem \ref{t.062501}]
    By the Yamada-Watanabe theorem \cite{kurtz2007yamada}, along with Proposition~\ref{p.070601} and Proposition~\ref{p.121001}, we know that for any $T>0$, the cutoff system \eqref{e.052504} has a unique pathwise solution $V\in L^2(\Omega;C([0,T;\mathbb \Hb^{\s}])\cap L^2(0,T;\Hb^{\s+\frac{s}{2}}))$. Define the stopping time
    \begin{align}\label{e.121001}
        \tau = \inf\{t\geq0: \|V\|_{\s}> \rho\}.
    \end{align}

    We first look at the case when $s\in(1,2)$. Let $C_0$ be the constant from the embedding $H^{\s}\subset W^{1,\infty}$. Assume that for some deterministic $M>0$, we have $\|V_0\|_{\s}\leq M$. Then, for any $\rho>2C_0M$, the stopping time $\tau$ is positive. Hence, $(V,\tau)$ is a local pathwise solution of \eqref{e.052502}. To generalize this result for $V_0\in L^2(\Omega,\Hb^{\sigma})$, we use a localization procedure.  For any $k\geq0$, set $V_0^k = V_0\mathbf{1}_{\{k\leq\|V_0\|_{\s}\leq k+1\}}$. Then the above argument yields a local pathwise solution $(V_k,\tau_k)$ with $\rho>2C_0(k+1)$. Then by defining 
    \begin{align*}
        &V=\sum_{k\geq0}V_k\mathbf{1}_{\{k\leq\|V_0\|_{\s}\leq k+1\}},\quad \tau = \sum_{k\geq0}\tau_k\mathbf{1}_{\{k\leq\|V_0\|_{\s}\leq k+1\}},
    \end{align*}
    we obtain the local  pathwise solution $(V,\tau)$ of \eqref{e.052502} with initial data $V_0\in L^2(\Omega,\Hb^{\sigma})$. To extend the solution to a maximal one, let $\mathcal{T}$ be the set of all stopping times corresponding to a local pathwise solution of \eqref{e.052502} with initial data $V_0$. By \cite[Chapter V, Section 18]{doob2012measure}, there exists a stopping time $\xi$ such that $\xi>\tau$ almost surely for any $\tau\in\mathcal{T}$, and there exists a sequence $\{\tau_n\}\subset \mathcal{T}$ satisfying $\tau_n \nearrow \xi$ almost surely. Let $(V_n,\tau_n)$ be the corresponding local pathwise solution and define 
    \begin{align*}
        V(t,\omega)=\lim_{n\to\infty} V_n(t\wedge\tau_n,\omega)\mathbf{1}_{[0,\xi)}(\omega). 
    \end{align*}
    Then $(V,(\tau_n)_{n\geq 1},\xi)$ is the desired maximal pathwise solution in the sense of Definition \ref{definition:pathwise-solution}. 

    For the case $s=1$, we require the cutoff parameter $\rho<\frac{1}{2C_{\s}}$, as stated in Proposition~\ref{t.062502}. Therefore for initial data $V_0$ with $\|V_0\|_{\s}<M:=\frac{1}{4C_0C_{\s}}$, we can always choose $\rho$ such that  $2C_0M<\rho<\frac{1}{2C_{\s}}$ to ensure a unique local pathwise solution $(V,\tau)$, where the stopping time $\tau>0$ corresponds to $\rho$ through \eqref{e.121001}. The extension to a maximal solution proceeds in the same manner as in the case $s \in (1,2)$. This completes the proof.
\end{proof}

\section{Remarks on the supercritical case}\label{s.121102}
In this section, we discuss the supercritical case ($s<1$). Two significant challenges arise in this setting. First, according to {\cite[Lemmas A.1 and A.3]{ghoul2022effect}}, the nonlinear term satisfies the following estimate: 
\begin{equation*}
    \begin{split}
        \Big|\Big\langle \LL^{\s} Q(V,V), \LL^{\s}  V \Big\rangle\Big| \leq  C_{\s} \|V\|_{\s} \|\LL^{\s+\frac{1}{2}} V\|^2. 
    \end{split}
\end{equation*}
Thus for $s<1$, the dissipation term $\|\LL^{\s+\frac{s}{2}} V\|^2$ is insufficient to control the nonlinearity. Indeed, as shown in \cite{abdo2024primitive}, the PE system with fractional dissipation $s<1$ is ill-posed in Sobolev spaces.
Therefore, one must work in the analytic class with a decaying analytic radius, similar to the approach in \cite{hu2023local}. 

The second difficulty arises when working with transport noise in the analytic class. Specifically, as shown in the example below, the cancellation of the highest-order terms involving the It\^o-Stratonovich corrector, as in \eqref{e.121004}, is no longer valid in the analytic setting unless the noise coefficient $b$ is independent of the spatial variable. If $b$ is spatially constant, the computations do not involve significant additional difficulties; thus, we omit them, referring interested readers to \cite{hu2023local}. Consequently, for $s<1$, the method developed in this work can only establish the local existence of pathwise solutions to \eqref{PE-system} in the analytic class when $b$ is spatially constant. The general case, where $b$ depends on spatial variables, requires alternative approaches and remains under investigation.

In the analytic setting, the cancellation terms involving the It\^o-Stratonovich corrector analogous to \eqref{e.121004} are given by
\[\big\langle \LL^{\s} e^{\tau \LL}(\Pc B_k)^2V, \LL^{\s} e^{\tau \LL}V\big\rangle + \|\LL^{\s}e^{\tau \LL}(\Pc B_k)V\|^2,\]
where $\tau>0$ is the analytic radius and $e^{\tau \LL}$ is defined in terms of the Fourier coefficients as:
\[(\widehat{e^{\tau\LL}f})_{\bk}  := e^{\tau|\bk|} \widehat{f}_{\bk}, \quad k \in 2\pi \mathbb Z^3.\]

The following example,  set on the 1D torus, demonstrates the cancellation of the terms with highest order Sobolev regularity, as in \eqref{e.121004}, generally does not hold when $\tau>0$. 
\begin{example}
    Let $b = 2\cos(x) = e^{ix}+e^{-ix}$ and denote by $a^*$ the complex conjugate of $a\in \mathbb C$. 
    A direct computation gives
    \begin{align*}
        I :&= \big\langle \LL^r e^{\tau \LL}  (b\partial_x(b\partial_x f)), \LL^r e^{\tau \LL}f\big\rangle\\
        & = -\sum_{k}2k^2|k|^{2r}e^{2\tau|k|}|\widehat{f_k}|^2-\sum_{k}(k-1)(k-2)|k|^{2r}e^{2\tau|k|}\widehat{f}_{k-2}\widehat{f}^*_k - \sum_{k} (k-1)k|k-2|^{2r}e^{2\tau |k-2|}\widehat{f}^*_{k-2}\widehat{f}_k, 
    \end{align*}
    and 
    \begin{align*}
        II:&=\big\langle \LL^r e^{\tau \LL}  (b\partial_x f), \LL^r e^{\tau \LL}(b\partial_x f)\big\rangle\\
        & = \sum_{k}((k+1)^{2r}e^{2\tau|k+1|}+(k-1)^{2r}e^{2\tau|k-1|})k^2|\widehat{f}_k|^2 + \sum_kk(k-2)|k-1|^{2r}e^{2\tau |k-1|}(\widehat{f}^*_{k-2}\widehat{f}_k+\widehat{f}_{k-2}\widehat{f}^*_k). 
    \end{align*}
    Using $\widehat{f}_k = \widehat{f}^*_{-k}$, we have
    \begin{align*}
        I+II = \sum_k \left(a(k)|\widehat{f}_k|^2 + b(k) \mathfrak{Re}( \widehat{f}_{k-2}\widehat{f}^*_k ) \right),
    \end{align*}
    where
    \begin{align*}
        a(k) &= ((k+1)^{2r}e^{2\tau|k+1|}+(k-1)^{2r}e^{2\tau|k-1|} - 2|k|^{2r}e^{2\tau|k|})k^2,\\
        b(k) &= 2k(k-2)|k-1|^{2r}e^{2\tau |k-1|} - (k-1)(k-2)|k|^{2r}e^{2\tau|k|} - k(k-1)|k-2|^{2r}e^{2\tau|k-2|}.
    \end{align*}
    When $\tau = 0$, all terms of degree greater than $2r$ cancel out. However, for $\tau >0$, such a cancellation breaks down due to the unequal exponential weights. This highlights a crucial difference between Sobolev-type estimates and analytic-type estimates. 
\end{example}

\section*{Acknowledgments}
R.H. was partially supported by a grant from the Simons Foundation (MP-TSM-00002783), and an ONR grant under \#N00014-24-1-2432. Q.L. was partially supported by an AMS-Simons travel grant.

\appendix
\section{Auxiliary lemmas}\label{section:auxlemma}

In this appendix, we summarize several lemmas that have been used repeatedly in our analysis. 

\begin{lemma}[see \cite{constantin2015long}]\label{l.071701}
    Let $s\geq 0$ and $f,g\in H^{s}\cap W^{1,\infty}$. We have 
    \begin{equation*}
        \begin{split}
            \|\LL^{s}(fg)\|_{L^2}\lesssim \|f\|_{L^{\infty}}\|g\|_{s}+\|g\|_{L^{\infty}}\|f\|_{s}.
        \end{split}
    \end{equation*}
    For $s>0$ we have 
    \begin{align*}
        \|\LL^s(fg)-f\LL^sg\|\lesssim \|\nabla f\|_{L^{\infty}}\|\LL^{s-1}g\|_{L^2}+\|\LL^sf\|_{L^{2}}\|g\|_{L^\infty}. 
    \end{align*}
\end{lemma}

We recall the following compactness result. For proofs, see \cite[Theorem 5]{simon1986compact} and \cite[Theorem 2.1]{flandoli1995martingale}, respectively.

\begin{lemma} 
    \label{lemma:aubin-lions}
    a) \emph{(Aubin-Lions-Simon Lemma).} Let $X_2 \subset X \subset X_1$ be Banach spaces such that the embedding $X_2 \hook \hook X$ is compact and the embedding $X \hook X_1$ is continuous. Let $p \in (1, \infty)$ and $\alpha \in (0, 1)$. Then, the following embedding is compact
    \[
        L^p(0, t; X_2) \cap W^{\alpha, p}(0, t; X_1) \hook \hook L^p(0, t; X).
    \]
    
    b) Let $X_2 \subset X$ be Banach spaces such that $X_2$ is reflexive and the embedding $X_2 \hook \hook X$ is compact. Let $\alpha \in (0, 1]$ and $p \in (1, \infty)$ be such that $\alpha p > 1$. Then, the following embedding is compact
    \[
        W^{\alpha, p}(0, t; X_2) \hook \hook C([0, t], X).	
    \]
\end{lemma}

We also recall the definitions of Sobolev spaces with fractional time derivative, see {\it e.g.}\ \cite{simon1990sobolev}. Let $X$ be a separable Hilbert space, $t > 0$, $p > 1$ and $\alpha \in (0, 1)$. The Sobolev space $W^{\alpha, p}(0, t; X)$ is defined as:
\[
	W^{\alpha, p}(0, t; X) =  \left\{ u \in L^p(0, t; X) \mid \int_0^t \int_0^t \frac{\vert u(s) - u(r) \vert_X^p}{|s - r|^{1+\alpha p}} \, dr \, ds < \infty \right\},
\]
with the norm:
\[
	\| u \|^{p}_{W^{\alpha, p}(0, t; X)} = \int_0^t \vert u(s) \vert^p_X \, ds + \int_0^t \int_0^t \frac{\vert u(s) - u(r) \vert_X^p}{|s - r|^{1+\alpha p}} \, dr \, ds.
\]

Finally, we recall the Burkholder-Davis-Gundy inequality. For a detailed proof, see \cite[Lemma 2.1]{flandoli1995martingale}. Here $L_2\left(\Uc, X\right)$ represents the Hilbert-Schmidt norm.
\begin{lemma}\label{bdg}
    Let $X$ be a separable Hilbert space. For $\Phi \in L^2\left(\Omega; L^2\left(0, T; L_2\left(\Uc, X\right)\right)\right)$, one has
    \begin{equation}\label{eq:bdg}
	\Eb \sup_{t \in \left[0, T\right]} \left| \int_0^t \Phi \, dW \right|^r_X \leq C_{r} \, \Eb \left( \int_0^T \| \Phi \|_{L_2(\Uc, X)}^2 \, dt \right)^{r/2}.
    \end{equation}
    Moreover, if $p \geq 2$ and $\Phi \in L^p\left(\Omega; L^p\left(0, T; L_2\left(\Uc, X\right)\right)\right)$, then
    \begin{equation}\label{eq:bdg.frac}
	\Eb \left| \int_0^\cdot \Phi \, dW \right|^p_{W^{\alpha, p}(0, T; X)} \leq c_{p} \, \Eb \int_0^T \| \Phi \|_{L_2(\Uc, X)}^p \, dt,
    \end{equation}
    for $\alpha \in [0, 1/2)$.
\end{lemma}

\section{Commutator estimates}\label{s.121103}
In this appendix, we omit the summation symbol for repeated indices. We first establish a result based on Lemma \ref{l.071701} and the following estimate obtained in \cite{chae2012generalized}: for any $s\in \Rb, \varepsilon>0$, and smooth functions $f,g$ over $\Tb^d$, there exists a constant $C_{s,\varepsilon}>0$ such that (where $\partial_j=\partial_{x_j}$)
\begin{align}\label{e.072201}
    \|[\LL^s\partial_j, g]f\|\leq C_{s,\varepsilon}\left(\|g\|_{H^{\frac{d}{2}+1+\varepsilon}}\|\LL^sf\|+\|g\|_{H^{\frac{d}{2}+1+s+\varepsilon}}\|f\|\right). 
\end{align}
It is worth noting that this result was originally proved in \cite{chae2012generalized} for the $\Rb^2$ case, and the generalization to $\Tb^d$ follows in the same manner.

\begin{lemma}\label{l.072102}
    Let $s>\frac52$ and $\alpha>-s$. Then for any smooth divergence-free vector field $b$ and function $V$ on $\Tb^3$, we have 
    \begin{align*}
    \|\LL^{\a} [\LL^{s},b\cdot\nabla]V\|\lesssim\|b\|_{s}\|V\|_{s+\a}+\|b\|_{\a+s}\|V\|_{s}. 
    \end{align*}
\end{lemma}
\begin{proof}
    By the divergence-free property of $b$, we first rewrite
    \begin{align*}
        \LL^{\a} [\LL^{s},b\cdot\nabla]V &= \LL^{\a+s}(b^j\partial_jV) - \LL^{\a}(b^j\LL^s\partial_jV)\\
        & = [\LL^{\a+s},b_j]\partial_jV-[\LL^{\alpha}\partial_j,b^j]\LL^sV. 
    \end{align*}
    For a fixed $j$, since $\alpha+s>0$ and $s>\frac52$, we apply Lemma \ref{l.071701} to obtain:
    \begin{align*}
        \|[\LL^{\a+s},b_j]\partial_jV\|&\lesssim \|\nabla b_j\|_{L^{\infty}}\|\LL^{\alpha+s-1}\partial_jV\|+\|\LL^{\a+s}b_j\|\|\partial_jV\|_{L^{\infty}}\\
        &\lesssim \|b\|_{s}\|V\|_{\alpha+s}+\|b\|_{\a+s}\|V\|_s.
    \end{align*}

    Next, since $s>\frac{3}{2}+1$, it follows from \eqref{e.072201} that 
    \begin{align*}
        \|[\LL^{\alpha}\partial_j,b^j]\LL^sV\|\lesssim \|b\|_{s}\|\LL^{\alpha+s}V\|+\|b\|_{\a+s}\|\LL^sV\|. 
    \end{align*}
    Combining the above estimates, we obtain the desired estimate. 

\end{proof}

We next establish a lemma on a double commutator. 
\begin{lemma}\label{l.072001}
    Assume $s>5/2$. Let $b$ be a smooth divergence-free vector field on $\Tb^3$ with zero mean. We have 
    \begin{align*}
        \|[[\LL^{s},b\cdot\nabla],b\cdot\nabla]f\|\lesssim \|b\|_{s+1}^2\|f\|_{s},
    \end{align*}
    for any real-valued smooth function $f$ with zero mean. 
\end{lemma}
\begin{proof}
    Since $b$ is divergence-free, we can rewrite:
    \begin{align}\label{e.072101}
        [\LL^s,b\cdot\nabla] f = \LL^s(\partial_j(b^jf)) - b^j\partial_j\LL^sf = [\LL^s\partial_j, b^j]f. 
    \end{align}
    Therefore, one has 
    \begin{align*}
        [[\LL^{s},b\cdot\nabla],b\cdot\nabla]f &= [\LL^s\partial_j, b^j]b^{\ell}\partial_{\ell}f-b^{\ell}\partial_{\ell}[\LL^s\partial_j, b^j]f\\
        & = [[\LL^s\partial_j, b^j],b^{\ell}](\partial_{\ell}f) + b^{\ell}[[\LL^s\partial_j, b^j],\partial_{\ell}]f:=I_1+I_2. 
    \end{align*}
    Below we analyze them individually. 
    
    Let $f_{\ell}=\partial_{\ell}f$. Since $b$ is divergence-free, one has 
    \begin{align*}
        I_1 &= [[\LL^s\partial_j, b^j],b^{\ell}](f_{\ell})\\
        & = \LL^s(b^j\partial_jb^{\ell}f_{\ell})+\LL^s(b^jb^{\ell}\partial_jf_{\ell}) - b^j\LL^s(\partial_jb^{\ell}f_{\ell})-b^j\LL^s(b^{\ell}\partial_jf_{\ell}) - b^{\ell}\LL^s(b^{j}\partial_jf_{\ell}) + b^{\ell}b^{j}\LL^s(\partial_jf_{\ell})\\
        & = [\LL^s, b^{j}]\partial_jb^{\ell}f_{\ell} + [\LL^s,b^{j}](b^{\ell}\partial_jf_{\ell}) - b^{\ell}[\LL^s,b^{j}](\partial_jf_{\ell})\\
        &= [\LL^s, b^{j}]\partial_jb^{\ell}f_{\ell} + [[\LL^s,b^{j}],b^{\ell}]\partial_jf_{\ell}.
    \end{align*}
    By Lemma \ref{l.071701} and Sobolev embeddings, we estimate
    \begin{align*}
        \|[\LL^s, b^{j}]\partial_jb^{\ell}f_{\ell}\|\lesssim \|\nabla b^j\|_{L^{\infty}}\|\LL^{s-1}(\partial_j b^{\ell}f_{\ell})\| + \|\LL^s b^j\|\|\partial_j b^{\ell}f_{\ell}\|_{L^{\infty}}\lesssim \|b\|_{s}^2\|f\|_s.
    \end{align*}
    For the second term in $I_1$, let $f_{j\ell} = \partial_jf_{\ell}$. Then
    \begin{align}\label{e.072001}
        [[\LL^s,b^{j}],b^{\ell}]f_{j\ell}= \LL^s(b^{j}b^{\ell}f_{j\ell}) - b^{j}\LL^s(b^{\ell}f_{j\ell}) -b^{\ell}\LL^s(b^{j}f_{j\ell})+b^{j}b^{\ell}\LL^sf_{j\ell}. 
    \end{align}
    By a straightforward application of the fractional Leibniz rule as in \cite[formula (1.6)]{li2019kato} with $s_1=2, p=2, g= f_{j\ell},$ we obtain that
    \begin{align*}
        \LL^s(b^{j}b^{\ell}f_{j\ell}) = \sum_{|\alpha|\leq 2}\frac{1}{\alpha!}\partial^{\alpha}(b^{j}b^{\ell})\LL^{s,\alpha}f_{j\ell} + \sum_{|\beta|< s-2}\frac{1}{\beta!}\partial^{\beta}f_{j\ell}\LL^{s,\beta}(b^{j}b^{\ell}) + E_1, 
    \end{align*} 
    where $\LL^{s,\alpha}$ is a multiplier of order $s-|\alpha|$ with $\LL^{s,0}=\LL^s$, and $\|E_1\|\lesssim \|\LL^2(b^{j}b^{\ell})\|_{L^{\infty}}\|\LL^{s-2}f_{j\ell}\|\lesssim \|b\|_{s+1}^2\|f\|_{s}$. Similarly, we have 
    \begin{align*}
        \LL^s(b^{\ell}f_{j\ell}) = \sum_{|\alpha|\leq 2}\frac{1}{\alpha!}\partial^{\alpha}b^{\ell}\LL^{s,\alpha}f_{j\ell} + \sum_{|\beta|< s-2}\frac{1}{\beta!}\partial^{\beta}f_{j\ell}\LL^{s,\beta}b^{\ell} + E_2, 
    \end{align*} 
    and  
    \begin{align*}
        \LL^s(b^{j}f_{j\ell}) = \sum_{|\alpha|\leq 2}\frac{1}{\alpha!}\partial^{\alpha}b^{j}\LL^{s,\alpha}f_{j\ell} + \sum_{|\beta|< s-2}\frac{1}{\beta!}\partial^{\beta}f_{j\ell}\LL^{s,\beta}b^{j} + E_3, 
    \end{align*}
    with 
    \[\|E_2\|\lesssim  \|\LL^2b^{\ell}\|_{L^{\infty}}\|\LL^{s-2}f_{j\ell}\| \lesssim \|b\|_{s+1}\|f\|_{s}, \text{ and } \|E_3\|\lesssim \|\LL^2b^{j}\|_{L^{\infty}}\|\LL^{s-2}f_{j\ell}\|\lesssim \|b\|_{s+1}\|f\|_{s}.\]
    In view of \eqref{e.072001}, the derivatives on $f_{j\ell}$ of orders $s-1$ and $s$ cancel out.  We therefore obtain by Sobolev embeddings that 
    \begin{align*}
        \|[[\LL^s,b^{j}],b^{\ell}]f_{j\ell}\|\lesssim \|b\|_{s+1}^2\|f\|_{s}. 
    \end{align*}

    For the commutator in $I_2$, we compute
    \begin{align*}
        [[\LL^s\partial_j, b^j],\partial_{\ell}]f &= \LL^s(b^j\partial_j\partial_{\ell}f) - b^j\LL^s\partial_j\partial_{\ell}f-\partial_{\ell}(\LL^s(b^j\partial_jf)) + \partial_{\ell}(b^j\LL^s\partial_jf)\\
        & = [\partial_{\ell}b^j,\LL^s]\partial_jf, 
    \end{align*}
    which together with Lemma~\ref{l.071701} lead to 
    \begin{align*}
        \|I_2\|\lesssim \|b^{\ell}\|_{L^{\infty}}\left(\|\nabla\partial_{\ell}b^j\|_{L^{\infty}}\|\LL^{s-1}\partial_jf\|+\|\LL^s\partial_{\ell}b^j\|\|\partial_jf\|_{L^{\infty}}\right)\lesssim \|b\|_{s+1}^2\|f\|_{s}. 
    \end{align*}
    This completes the proof. 

\end{proof}

Next, we provide a lemma concerning commutator estimates in Sobolev spaces with negative exponents. Since we could not locate a corresponding result in the literature, we provide the proof for completeness. 
\begin{lemma}\label{l.062701}
	Let $s\geq 1, 0\leq \a\leq s$ and $\beta>\frac32$.  For any smooth divergence-free vector field $b$ and function $V$ on $\Tb^3$, we have 
	\begin{align*}
		\|\LL^{-\a} [\LL^{s},b\cdot\nabla] V\|\lesssim_{s,\a,\b}\|b\|_{s+\a+\b}\|V\|_{s-\a}. 
	\end{align*}
\end{lemma}
\begin{proof}
	Using the Minkowski inequality and the fact that 
	\[|a^s-b^s|\lesssim_{s}|a-b|(b^{s-1}+|a-b|^{s-1}),\]
	for $a,b\geq 0$,  we have
	\begin{align*}
		\begin{split}
			\| [\LL^{s},b\cdot\nabla] V\|_{-\a} 
			&= \left(\sum_{k}|k|^{-2\a}\left|\sum_j\widehat{b}_{k-j}\cdot ij \widehat{V}_j\left(|k|^{s}-|j|^{s}\right)\right|^2\right)^{\frac12}\\
			&\lesssim_{s} \left(\sum_{k}|k|^{-2\a}\left(\sum_j|\widehat{b}_{k-j}||j| |\widehat{V}_j||k-j|\left(|j|^{s-1}+|k-j|^{s-1}\right)\right)^2\right)^{\frac12}\\
			&\lesssim_{s} \left(\sum_{k}|k|^{-2\a}\left(\sum_j|\widehat{b}_{k-j}||\widehat{V}_j||k-j||j|^{s}\right)^2\right)^{\frac12}\\
			&\qquad + \left(\sum_{k}|k|^{-2\a}\left(\sum_j|\widehat{b}_{k-j}||j| |\widehat{V}_j||k-j|^{s}\right)^2\right)^{\frac12}:=I_1+I_2. 
		\end{split}
	\end{align*}
	From $|j|\leq |k|+|k-j|$ and Young's convolution inequality, we deduce: 
	\begin{align}\label{e.L062701}
		\begin{split}
			I_1 &= C_s\left(\sum_{k}\left(\sum_j|\widehat{b}_{k-j}||\widehat{V}_j||k-j|^{\a+1}|j|^{s-\a}\frac{|j|^{\a}}{|k|^{\a}|k-j|^{\a}}\right)^2\right)^{\frac12}\\
			&\lesssim_{s,\a}\left(\sum_{k}\left(\sum_j|\widehat{b}_{k-j}||\widehat{V}_j||k-j|^{\a+1}|j|^{s-\a}\right)^2\right)^{\frac12}\\
			&\lesssim_{s,\a}\left(\sum_j|\widehat{b}_{j}||j|^{\a+1}\right)\|V\|_{s-\a}\lesssim_{s,\a,\b}\|b\|_{\a+\b+1}\|V\|_{s-\a}. 
		\end{split}
	\end{align}
    Using Minkowski's inequality, we split $I_2$ into three parts: 
	\begin{align*}
		\begin{split}
			I_2 &\lesssim_{s}\left(\sum_{k}|k|^{-2\a}\left(\sum_{|j|\leq \frac12|k-j|}|\widehat{b}_{k-j}||j| |\widehat{V}_j||k-j|^{s}\right)^2\right)^{\frac12}+\left(\sum_{k}|k|^{-2\a}\left(\sum_{|j|\geq 2|k-j|}|\widehat{b}_{k-j}||j| |\widehat{V}_j||k-j|^{s}\right)^2\right)^{\frac12}\\
            &\qquad +\left(\sum_{k}|k|^{-2\a}\left(\sum_{\frac12|k-j|<|j|< 2|k-j|}|\widehat{b}_{k-j}||j| |\widehat{V}_j||k-j|^{s}\right)^2\right)^{\frac12}:=I_{21}+I_{22}+I_{23}. 
		\end{split}
	\end{align*}
For $I_{21}$, noting when $|j|\leq \frac12|k-j|$, there holds 
    \[\frac12|k-j|\leq |k|\leq \frac32|k-j|.\]
    Hence, by the assumption $s\geq 1$ and Young's convolution inequality we have  
    \begin{align*}
        \begin{split}
            I_{21} &\lesssim_{s,\a} \left(\sum_{k}\left(\sum_{|j|\leq \frac12|k-j|}|\widehat{b}_{k-j}||j| |\widehat{V}_j||k-j|^{s-\a}\right)^2\right)^{\frac12}\\
            &\lesssim_{s,\a} \left(\sum_{k}\left(\sum_{|j|\leq \frac12|k-j|}|\widehat{b}_{k-j}||\widehat{V}_j||k-j|^{s}|j|^{s-\a}\right)^2\right)^{\frac12}\\
            &\lesssim_{s,\a} \sum_{j}|\widehat{b}_{j}||j|^{s}\|V\|_{s-\a}\lesssim_{s,\a,\b} \|b\|_{s+\b}\|V\|_{s-\a}.
        \end{split}
    \end{align*}
    Similarly, we estimate 
    \[I_{22}\lesssim_{s,\a,\b} \|b\|_{s+\b}\|V\|_{s-\a}.\]
    For $I_{23}$, since $|k-j|\sim|j|$, we have 
    \begin{align*}
        \begin{split}
            I_{23}&\lesssim_{s}\left(\sum_{k}\left(\sum_{\frac12|k-j|<|j|< 2|k-j|}|\widehat{b}_{k-j}||k-j| |\widehat{V}_j||j|^{s}\right)^2\right)^{\frac12}
            \\
            &\lesssim_{s}\left(\sum_{k}\left(\sum_{\frac12|k-j|<|j|< 2|k-j|}|\widehat{b}_{k-j}||k-j|^{1+\a} |\widehat{V}_j||j|^{s-\a}\right)^2\right)^{\frac12}
            \\
            &\lesssim_{s,\a} \left(\sum_j|\widehat{b}_{j}||j|^{\a+1}\right)\|V\|_{s-\a}\lesssim_{s,\a,\b}\|b\|_{\a+\b+1}\|V\|_{s-\a}.
        \end{split}
    \end{align*}
    Therefore 
    \[I_2\lesssim_{s,\a,\b}\|b\|_{s+\a+\b}\|V\|_{s-\a},\]
    which together with \eqref{e.L062701} completes the proof. 

\end{proof}

Finally, we give a lemma regarding commutators involving the hydrostatic Leray projector $\Pc$  defined in \eqref{e.062501}. 
\begin{lemma}\label{l.072101}
    Assume $s>2$. Let $b=(b^1,b^2,b^3)\in C^{\infty}(\Tb^3,\Rb^3)$ be divergence-free and $\varphi\in C^{\infty}(\Tb^3,\Rb^2)$ such that $\Pc\varphi = \varphi$. Then we have 
    \[\|[\Pc,b\cdot\nabla]\varphi\|_{s}\lesssim \|b\|_{s+1}\|\varphi\|_{s}+\|\partial_z(b^1,b^2)\|_{s-1}\|\varphi\|_{s+1}.\]
    In addition, if $s>\frac52$, then one has
    \begin{align*}
        \|\LL^{-\frac12}[\LL^s,b\cdot\nabla][\Pc,b\cdot\nabla] \varphi\|\lesssim \|b\|_{s+3}^2\|\varphi\|_{s-\frac12}+\|b\|_{s+3}\|\partial_z(b^1,b^2)\|_{s-\frac32}\|\varphi\|_{s+\frac12}.
    \end{align*}
\end{lemma}
\begin{proof}
    To simplify the notation, let $w=-\int_0^z \nabla_h \cdot \varphi(x',\tilde z) d\tilde z$
    so that $\nabla_h \cdot \varphi + \partial_z w = 0$. From the definition in \eqref{e.062501}, we have 
    \begin{align*}
        [\Pc,b\cdot\nabla]\varphi = \Pc(b\cdot\nabla\varphi) - b\cdot\nabla\Pc\varphi = (I-\mathbb Q)(b\cdot\nabla\varphi) - b\cdot\nabla\varphi = -\mathbb Q(b\cdot\nabla\varphi). 
    \end{align*}
    Note that (recall from \eqref{e.062501} that $\overline{f}$ is the average of $f$ with respect to the $z$ variable)
    \begin{align*}
        \mathbb Q(b\cdot\nabla\varphi) = \nabla_h\Delta_h^{-1}\nabla_h\cdot(\overline{b\cdot\nabla\varphi}) 
    \end{align*}
    depends only on the horizontal variables $(x_1,x_2)$.
    Using integration by parts, the divergence-free property of $b,(\varphi,w)$ and the periodicity, we compute 
    \begin{align*}
        \nabla_h\cdot\overline{b\cdot\nabla\varphi} &=\int_0^1\nabla_h\left((b^1,b^2)\cdot\nabla_h\right)\cdot\varphi dz + \int_0^1(b^1,b^2)\cdot\nabla_h(\nabla_h\cdot\varphi)dz + \nabla_h\cdot\int_0^1b^3\partial_z\varphi dz\\
        & = \int_0^1\nabla_h\left((b^1,b^2)\cdot\nabla_h\right)\cdot\varphi dz+\int_0^1\partial_z(b^1,b^2)\cdot\nabla_hw dz+ \nabla_h\cdot\int_0^1\varphi\nabla_h\cdot(b^1,b^2) dz.
    \end{align*}
    
    When $s>2$, the space $H_{h}^{s-1}$ is a Banach algebra in the $2D$ case. Since $\mathbb Q(b\cdot\nabla\varphi)$ depends only on the horizontal variables $(x_1,x_2)$, applying Minkowski's inequality and H\"older inequality, we have 
    \begin{align*}
        \|\mathbb Q(b\cdot\nabla\varphi)\|_{s} = \|\mathbb Q(b\cdot\nabla\varphi)\|_{H^s_{h}}
        &\lesssim\| \nabla_h\cdot\overline{b\cdot\nabla\varphi}\|_{H^{s-1}_h}\\
        &\lesssim \int_0^1 \left(\|b(\cdot,z)\|_{H^{s+1}_h}\|\varphi(\cdot,z)\|_{H^{s}_h}+\|\partial_z(b^1,b^2)(\cdot,z)\|_{H^{s-1}_h}\|w(\cdot,z)\|_{H^{s}_h}\right) dz
        \\
        &\lesssim \|b\|_{s+1}\|\varphi\|_{s}+\|\partial_z(b^1,b^2)\|_{s-1}\|\varphi\|_{s+1}. 
    \end{align*}
    
    Combining this with Lemma \ref{l.062701}, when $s>\frac52$ we conclude
    \begin{align*}
        \|\LL^{-\frac12}[\LL^s,b\cdot\nabla][\Pc,b\cdot\nabla] \varphi\|
        &\lesssim \|b\|_{s+3}\|[\Pc,b\cdot\nabla]\varphi\|_{s-\frac12}\\
        &\lesssim \|b\|_{s+3}\left( \|b\|_{s+\frac12}\|\varphi\|_{s-\frac12}+\|\partial_z(b^1,b^2)\|_{s-\frac32}\|\varphi\|_{s+\frac12}\right)\\
        &\lesssim \|b\|_{s+3}^2\|\varphi\|_{s-\frac12}+\|b\|_{s+3}\|\partial_z(b^1,b^2)\|_{s-\frac32}\|\varphi\|_{s+\frac12}. 
    \end{align*}
\end{proof}
\bibliographystyle{plain}
\bibliography{Reference}

\end{document}